\documentclass[10pt]{article}

\usepackage{amsmath}
\usepackage{amsfonts}
\usepackage{amssymb}
\usepackage{amscd}
\usepackage{amsthm}
\usepackage{amsbsy}
\usepackage{graphicx}
\usepackage{bm}

\newcommand{\map}[1]{\xrightarrow{#1}}
\newcommand{\iso}{\cong}

\newcommand{\Hom}{\mathrm{Hom}}
\newcommand{\Aut}{\mathrm{Aut}}
\newcommand{\End}{\mathrm{End}}
\newcommand{\Spec}{\mathrm{Spec}}
\newcommand{\Q}{\mathbb Q}
\newcommand{\Z}{\mathbb Z}
\newcommand{\R}{\mathbb R}
\newcommand{\C}{\mathbb C}
\newcommand{\F}{\mathbb F}
\newcommand{\A}{\mathbb A}
\newcommand{\co}{\mathcal O}
\newcommand{\alg}{\mathrm{alg}}
\newcommand{\ord}{\mathrm{ord}}
\newcommand{\length}{\mathrm{length}}

\newcommand{\action}{\bullet}
\newcommand{\inv}{\mathrm{inv}}
\newcommand{\Diff}{\mathrm{Diff}}
\newcommand{\Sppt}{\mathrm{Sppt}}
\newcommand{\cmdeg}{\mathrm{deg}_{\mathrm{CM}}}
\newcommand{\CM}{\mathrm{CM}}

\newcommand{\Sch}{\mathfrak{S}}

\newcommand{\tr}{\mathrm{Tr}}

\newcommand{\SL}{\mathrm{SL}}

\newcommand{\kzxz}[4]{\left(\begin{smallmatrix} #1 & #2 \\ #3 & #4\end{smallmatrix}\right) }

\newcommand{\SO}{\mathrm{SO}}
\newcommand{\CL}{\mathrm{CL}}
\newcommand{\norm}{\operatorname{N}}
\newcommand{\ff}{\mathrm{if  }}

\newcommand{\vol}{\mathrm{Vol}}

\input xy
\xyoption{all}

\date{}
\title{Singular Moduli Refined }
\author{Benjamin Howard \thanks{Supported in part by NSF grant DMS-0901753.}\and
Tonghai Yang\thanks{Supported in part by grants DMS-0555503, NSFC-10628103.}}

\begin{document}

\maketitle


\theoremstyle{plain}
\newtheorem{Thm}{Theorem}[section]
\newtheorem{Prop}[Thm]{Proposition}
\newtheorem{Lem}[Thm]{Lemma}
\newtheorem{Cor}[Thm]{Corollary}
\newtheorem{Conj}[Thm]{Conjecture}
\newtheorem{BigThm}{Theorem}
\newtheorem{BigCor}[BigThm]{Corollary}

\theoremstyle{definition}
\newtheorem{Def}[Thm]{Definition}
\newtheorem{Hyp}[Thm]{Hypothesis}

\theoremstyle{remark}
\newtheorem{Rem}[Thm]{Remark}
\newtheorem{Ques}[Thm]{Question}

\numberwithin{equation}{section}
\renewcommand{\theBigThm}{\Alph{BigThm}}
\renewcommand{\theBigCor}{\Alph{BigCor}}



\thispagestyle{empty}

\begin{abstract}
\vskip 3mm\footnotesize{

\vskip 4.5mm
\noindent
In this paper, we give a refinement of the work of Gross and Zagier on singular moduli. 
Let $K_1$ and $K_2$ be two imaginary quadratic fields with relatively prime discriminants $d_1$ and $d_2$, and let 
$F=\Q(\sqrt {d_1 d_2})$. Hecke constructed a Hilbert modular Eisenstein series over $F$ of weight $1$ 
whose functional equation forces it to vanish at $s=0$. For CM elliptic curves $E_1$ and $E_2$  with complex 
multiplication  by $\co_{K_1}$ and $\co_{K_2}$, we define an $\co_F$-module structure and  $\co_F$-quadratic form 
$\deg_\CM$ on $\Hom(E_1, E_2)$, which is totally positive definite and satisfies $\tr_{F/\Q} \deg_{\CM}=\deg$. 
For each totally positive  $\alpha \in F$ we consider the moduli stack $\mathcal{X}_\alpha$ of triples $(E_1, E_2, j)$ 
with  $E_i$ as above  and $j \in \Hom(E_1,E_2)$ with $\deg_{\CM}(j) =\alpha$.  
We prove that $\mathcal{X}_\alpha$  has dimension $0$, and that its arithmetic degree is equal to the 
$\alpha$-th Fourier coefficient of the central derivative of Hecke's Eisenstein series.

\vspace*{2mm}

\vspace*{2mm}
}

\end{abstract}
%

\section{Introduction}


Let $K_1$ and $K_2$ be non-isomorphic quadratic imaginary fields with
discriminants $d_1$ and $d_2$, respectively,  and set
$K=K_1\otimes_\Q K_2$.  Let $F$ be the real quadratic subfield of $K$, set
$D=\mathrm{disc}(F)$, and let $\mathfrak{D}\subset \co_F$ be the
different of $F/\Q$.   Let $x\mapsto \overline{x}$ denote complex
conjugation on $K$ and set $\mathbf{w}_i=|\co_{K_i}^\times|$.
 Let $\chi$ be  the quadratic Hecke character of $F$ associated to $K$, and let
$\sigma_1$ and $\sigma_2$ be the  two  real embeddings of $F$.
We assume throughout this paper that $\gcd(d_1,d_2)=1$ so  that $K/F$ is unramified
at all finite places, and  $\co_{K_1}\otimes_\Z\co_{K_2}$ is the maximal order in $K$.

Almost one hundred years ago,  Hecke
constructed his famous Eisenstein series (see \cite[(7.2)]{GZSingular})  of parallel weight $1$ for $\SL_2(\co_F)$
   \begin{align*} \label{Hecke2}
   &E^*(\tau_1, \tau_2, s)
    =D^{\frac{s+1}2} \left(\pi^{-\frac{s+2}2}\Gamma\left(\frac{s+2}2\right)\right)^2
\sum_{[\mathfrak a] \in \CL(F)} \chi(\mathfrak a) \norm(\mathfrak
a)^{1+s}
\\
 &\quad \times
 \sum_{(0, 0) \ne (m, n) \in \mathfrak a^2/\co_F^\times}
 \frac{ (v_1 v_2)^{\frac{s}2}}{(m (\tau_1, \tau_2)+n) |m (\tau_1, \tau_2)+n|^s}.
    \end{align*}
    Here $\CL(F)$ is the ideal class group of $F$,  $[\mathfrak{a}]$ denotes the class of the fractional ideal $\mathfrak{a}$,  and
    $$
m (\tau_1, \tau_2)+n= (\sigma_1(m) \tau_1 + \sigma_1(n))(\sigma_2(m)
\tau_2 +\sigma_2(n)). $$
Hecke showed that this sum, convergent for $\mathrm{Re}(s) \gg 0$,  has meromorphic continuation to all
$s$ and defines   a non-holomorphic Hilbert modular form of weight $1$ for     $\SL_2(\co_F)$ which
is holomorphic (in the variable $s$) in a neighborhood of $s=0$. The value $E^*(\tau_1, \tau_2, 0)$ at $s=0$ is a holomorphic     Hilbert modular form
of weight $1$ (Hecke's trick). He further computed the Fourier expansion of
this holomorphic modular form.   Unfortunately, he missed  a sign in the
calculation, and it turns out that $E^*(\tau_1, \tau_2, 0) = 0$ identically.  In the early 1980's,  Gross and  Zagier took
advantage of this fact to compute the central derivative at
$s=0$, and found that the Fourier coefficients of the diagonal restriction
to the upper half plane  are very closely related
to the  factorization of  singular moduli      (see \cite{GZSingular}).  Their result can be  rephrased
(see \cite[Section 3]{Yam=1} or Corollary \ref{cor1.2} below for more details) in terms of arithmetic intersections as follows:
if $\mathcal{E}$ is the moduli stack of elliptic curves over $\Z$-schemes then
the $m$-th Fourier coefficient of $E^{*,\prime}(\tau, \tau, 0)$  is the arithmetic intersection on
$\mathcal{E}\times\mathcal{E}$ of the $m$-th
Hecke correspondence with the  codimension two cycle of points representing pairs
$(\mathbf{E}_1,\mathbf{E}_2)$ of elliptic curves with complex multiplication by  
$\co_{K_1}$ and $\co_{K_2}$, respectively.  One naturally asks, for $\alpha\in F^\times$
what is the arithmetic meaning of the $\alpha$-th Fourier
coefficient of the central derivative $E^{*, \prime}(\tau_1,\tau_2, 0)$ itself, before one restricts to the
diagonal $\tau_1=\tau_2$? In another word, is there an arithmetic Siegel-Weil formula
(in the sense of \cite{KuMSRI} or \cite{KRYBook})
for this Hecke Eisenstein series? The purpose of this paper is to answer this question positively.

Let $\mathcal X$ be the algebraic stack over $\mathbb Z$ representing the functor that
 assigns to every scheme $S$ the  category  $\mathcal{X}(S)$  of   pairs
 $(\mathbf{E}_1, \mathbf{E}_2)$ in which  each $\mathbf{E}_i=(E_i,\kappa_i)$ consists of  an  elliptic curve 
 $E_i$ over $S$ and an action $\kappa_i:\co_{K_i}\map{}\End(E_i)$.     For an object  
 $(\mathbf E_1, \mathbf E_2)$ of  $\mathcal X(S)$ let
 $$
 L(\mathbf E_1, \mathbf E_2)  =\Hom(  E_1,  E_2)
 $$
be the $\mathbb Z$-module of  homomorphisms  from $ E_1$ to $E_2$,  equipped with the quadratic form
$\deg$.    Let $[\, , \,]$ be the  bilinear form associated to $\deg$.
The maximal order  $$\co_K  =\co_{K_1} \otimes_{\mathbb Z} \co_{K_2}$$
acts on $L(\mathbf E_1, \mathbf E_2) $ by
$$
(t_1 \otimes t_2) \action  j=   \kappa_2 (t_2) \circ   j \circ  \kappa_1 (\bar{t}_1)
$$
($t_i \in \co_{K_i}$) making $L(\mathbf E_1, \mathbf E_2)$ into  an    $\co_K$-module.  The
action satisfies
$$
[t \action  j_1, j_2] = [j_1, \bar t \action  j_2]
$$
  for all $t\in \co_K$, and it follows that if we view $L(\mathbf{E}_1,\mathbf{E}_2)$ as an $\co_F$-module then
   there is a unique $\co_{F}$-bilinear form
  $$
  [\ ,\ ]_{\CM}: L(\mathbf E_1, \mathbf E_2)\times L(\mathbf E_1, \mathbf E_2) \map{}\mathfrak{D}^{-1}
  $$
  satisfying $[j_1,j_2] = \tr_{F/\Q} [j_1,j_2]_\CM$.  If $\cmdeg$ is the totally positive definite $F$-quadratic form on
  $L(\mathbf{E}_1,\mathbf{E}_2)\otimes_\Z\Q$ corresponding to $[\ ,\ ]_\CM$ then
  $$
  \deg (j) = \tr_{F/\mathbb Q} \deg_{\CM} (j).
  $$
For any  $\alpha \in F^\times$  let $\mathcal{X}_{\alpha}$ be the algebraic stack representing  the functor that  assigns to a
scheme  $S$ the category $\mathcal{X}_\alpha(S)$ of  triples $(\mathbf    E_1, \mathbf E_2, j)$  in which
$(\mathbf E_1, \mathbf E_2)$ is an object of $\mathcal X(S)$ and $j \in L(\mathbf E_1, \mathbf E_2)$ with   $\deg_{\CM}(j)=\alpha$.
It is clear that $\mathcal X_\alpha$ is empty unless  $\alpha$ is totally positive.

For $\alpha\in F^\times$ totally positive  define the \emph{Arakelov degree}
\begin{equation}\label{arakelov degree}
\deg(\mathcal{X}_\alpha) = \sum_p \log(p) \sum_{x\in [\mathcal{X}_\alpha(\F_p^\alg)] }
e_x^{-1} \cdot\length( \co^\mathrm{sh}_{  \mathcal{X}_\alpha, x  } )
\end{equation}
where  $[\mathcal{X}_\alpha(S)]$ is the set of     isomorphism classes of objects in the category
$\mathcal{X}_\alpha(S)$,  $\co^\mathrm{sh}_{  \mathcal{X}_\alpha, x } $ is the
strictly Henselian local ring of $\mathcal{X}_\alpha$ at $x$, and
$e_x$ is the order of the automorphism group of the triple
$(\mathbf{E}_1,\mathbf{E}_2,j)$ corresponding to $x$.  Define $
\Diff(\alpha) $ to be the set of finite primes $\mathfrak
p$ of $F$ satisfying
$$
\chi_{\mathfrak p}(\alpha \mathfrak D) =-1.
$$
Using the product formula $\prod_w\chi_w(\alpha\sqrt{D})=1$ and the fact that $\mathfrak{D}=\sqrt{D}\co_F$,
we see that  $\Diff(\alpha)$ has  odd cardinality, and in particular is nonempty.
If $\mathfrak{b}$ is a fractional $\co_{F}$-ideal  we define $\rho(\mathfrak{b})$ to be the number of
ideals $\mathfrak{B}\subset \co_{K}$ satisfying $N_{K/F}(\mathfrak{B}) = \mathfrak{b}$.  If
$\ell$ is a rational prime  we define $\rho_\ell(\mathfrak{b})$ to be the number of ideals
$\mathfrak{B}\subset \co_{K,\ell}$ satisfying
$N_{K_\ell/F_{\ell}}(\mathfrak{B}) = \mathfrak{b}_\ell$.  Thus
\begin{equation}\label{product ideals}
\rho(\mathfrak{b}) = \prod_\ell \rho_\ell(\mathfrak{b}).
\end{equation}

For the proof of the following theorem see Section \ref{ss:arithmetic section}.

\begin{BigThm}\label{geometric theorem}
Suppose $\alpha\in F$ is totally positive. If   $\alpha\in \mathfrak{D}^{-1}$ and  
$\Diff(\alpha)=\{\mathfrak{p}\}$  then $\mathcal{X}_\alpha$
has dimension zero, is supported in characteristic $p$ (the rational prime below $\mathfrak{p}$), and satisfies
$$
\deg(\mathcal{X}_\alpha) = \frac{1}{2} \ord_\mathfrak{p}(\alpha\mathfrak{p}\mathfrak{D})
 \rho(\alpha\mathfrak{D}\mathfrak{p}^{-1})\cdot \log(p).
$$
If $\alpha \notin \mathfrak D^{-1}$ or if $|\Diff(\alpha)| >1$, then   $\mathcal{X}_\alpha=\emptyset$.
\end{BigThm}

The functional equation forces $E^*(\tau_1, \tau_2, 0) =0$, and the central derivative has a Fourier expansion
$$
E^{*, \prime}(\tau_1, \tau_2, 0)
 = \sum_{\alpha \in \mathfrak D^{-1}} a_\alpha(v_1, v_2) \cdot q^\alpha
 $$
 where $v_i = \mathrm{Im}(\tau_i)$, $e(x)=e^{2\pi i x}$, and $q^\alpha =e( \sigma_1(\alpha) \tau_1 + \sigma_2(\alpha)
 \tau_2)$.

\begin{BigThm} \label{theo1.2} 
Suppose $\alpha \in F$ is totally positive. If  $\alpha\in\mathfrak{D}^{-1}$  and
$\Diff(\alpha)=\{\mathfrak{p}\}$, then $a_\alpha=a_{\alpha}(v_1, v_2)$ is independent of $v_1, v_2 $,  and
 $$
 a_\alpha =2  \ord_{\mathfrak p} (\alpha \mathfrak p \mathfrak D)  \rho(\alpha
 \mathfrak D \mathfrak p^{-1}) \cdot \log (p).
 $$
Here $p$ is the rational prime below $\mathfrak{p}$.  If $\alpha \notin \mathfrak D^{-1}$ or if $|\Diff(\alpha)| >1$,
then $a_\alpha =0$.
\end{BigThm}

Theorem \ref{theo1.2} is stated in a different form in \cite{GZSingular}, but without proof. We will give a sketch of the
proof  in Section \ref{sect:explicit}. Combining the above theorems we obtain the following.

\begin{BigThm} \label{maintheo}
Assume $\alpha \in F$ is totally positive.  Then $\mathcal  X_{\alpha}$ is a stack of dimension zero and
$$
4 \cdot  \deg ( \mathcal X_{\alpha} )=   a_{\alpha}
$$
where $a_\alpha$ is the $\alpha$-th Fourier coefficient of $E^{*,\prime}(\tau_1, \tau_2, 0)$.
\end{BigThm}

In Section  \ref{sect:concept} we    give a slightly different and more
conceptual proof of Theorem \ref{maintheo}, based on the Siegel-Weil formula, which we now outline.
Fix a totally positive $\alpha\in F$. Assume that $\alpha \in \mathfrak D^{-1}$ and 
$\Diff(\alpha) =\{ \mathfrak p\}$ (otherwise both sides of the equality of  Theorem \ref{maintheo} are equal to zero).
Let $p$ be the rational prime lying below $\mathfrak p$, so that  $\mathcal X_\alpha$ is supported in characteristic $p$.
It is proved in Theorem \ref{Thm:local length} that
$$
\length( \co^\mathrm{sh}_{  \mathcal{X}_\alpha, x  } )=\nu_\mathfrak p(\alpha)
$$
for some  explicit number $\nu_\mathfrak p(\alpha)$ independent of  $x$, and so
 $$
 \deg(\mathcal X_\alpha) = \nu_\mathfrak p(\alpha) \log (p) \sum_{x\in [\mathcal{X}_\alpha(\F_p^\alg)] }
e_x^{-1} .
 $$
Next, applying the Siegel-Weil formula, one can prove that the summation on the right is, up to a factor of $1/4$,
the  $\alpha$-th Fourier coefficient of the value at $s=0$ of a coherent (in the sense of Kudla \cite{KuAnnals})
Eisenstein series $E_\alpha^*(\tau, s, \phi^{(\mathfrak p)})$;
see Proposition \ref{Prop:counting} and the argument at the end of the paper.  Therefore
$$
 \deg(\mathcal X_\alpha) \cdot q^\alpha =\frac{1}{4}  \nu_\mathfrak p(\alpha) \log (p) \cdot  E_\alpha^*(\tau, 0, \phi^{(\mathfrak p)}).
$$
On the other hand, $\Diff(\alpha) =\{\mathfrak p\}$ implies
$$
E_\alpha^{*, \prime}(\tau, 0) =
\frac{W_{\alpha, \mathfrak p}^{*, \prime}(1, 0)}{W_{\alpha, \mathfrak p}^{*}(1, 0, \phi_{\mathfrak p}^{(\mathfrak p)})}
E_\alpha^*(\tau, 0, \phi^{(\mathfrak p)}),
$$
where $W_{\alpha, \mathfrak p}^*(1, s)$ is the local Whittaker function of $E^*(\tau, s)$ at $\mathfrak p$,
and similarly for $W_{\alpha, \mathfrak p}^*(1,s, \phi^{(\mathfrak p)})$. Finally, explicit calculation shows
(see Proposition \ref{lem3.11} and the argument at the end of the paper)
$$
\frac{W_{\alpha, \mathfrak p}^{*, \prime}(1, 0)}{W_{\alpha, \mathfrak p}^{*}(1, 0, \phi_{\mathfrak p}^{(\mathfrak p)})}
=\nu_\mathfrak p(\alpha) \log(p).
$$
Therefore
$$
E_\alpha^{*, \prime}(\tau, 0) = \nu_\mathfrak p(\alpha) \log p \cdot E_\alpha^*(\tau, 0, \phi^{(\mathfrak p)})
 =4 \deg(\mathcal  X_\alpha)  \cdot q^\alpha .
$$
This gives a proof of  Theorem \ref{maintheo} without explicitly counting the number of points in
$[\mathcal{X}_\alpha(\F_p^\alg)]$, and without explicitly computing the Fourier coefficient $a_\alpha$.

By  Theorem \ref{maintheo}, one sees that the generating function
\begin{equation*}
\phi(\tau) =\sum_{ \substack{ \alpha \in \mathfrak D^{-1} \\ \alpha \gg 0  }} \deg (\mathcal
X_\alpha)  \cdot q^\alpha
\end{equation*}
is the holomorphic part of a non-holomorphic  Hilbert modular form
of weight $1$ for $\SL_2(\co_F)$, namely $E^{*, \prime}(\tau_1, \tau_2,
0)$. One can also view the theorem as an arithmetic Siegel-Weil
formula in the sense of  \cite{KuMSRI} and \cite{KRYBook}---giving
an arithmetic interpretation of the central derivative of the
incoherent Eisenstein  series $E^*(\tau_1, \tau_2, s)$.

We now explain in what sense Theorem \ref{geometric theorem} is a  refinement of the
earlier work of Gross and Zagier on singular moduli.   For a positive integer $m$ let $\mathcal T_m$ be the
algebraic stack representing the functor that assigns to a scheme $S$ the category of all
triples $(\mathbf E_1, \mathbf E_2, j)$ where $(\mathbf E_1, \mathbf E_2)$ is an object of
$\mathcal X(S) $  and $j \in\Hom(E_1,E_2)$ satisfies $\deg (j) =m$.  Directly from the moduli problems we have
$$
\mathcal T_m = \bigsqcup_{\substack{\alpha \in F \\ \tr_{F/\mathbb Q}
(\alpha) =m}} \mathcal X_\alpha.
$$
Combining this decomposition with the formula for $\deg(\mathcal{X}_\alpha)$ of Theorem \ref{geometric theorem}
one finds (see Corollary \ref{Cor:GZ}) a formula for $\deg(\mathcal{T}_m)$.  This formula is precisely the main
result of \cite{GZSingular}.

\begin{BigCor}[Gross-Zagier] \label{cor1.2}
For any positive integer $m$ we have
$$
\deg (\mathcal T_m)  = \frac{1}2 \sum_{\substack{ \alpha \in \mathfrak
D^{-1} \\ \tr_{F/\mathbb Q} (\alpha) =m \\ a \gg 0}} \sum_{p} \log p
\sum_{\mathfrak p| p} \ord_\mathfrak p (\alpha \mathfrak D \mathfrak
p) \rho(\alpha \mathfrak D \mathfrak p^{-1})
$$
where the middle summation is over those rational primes $p$ that are nonsplit in both $K_1$ and $K_2$.
\end{BigCor}

This work grew out of the authors' attempts to understand Gross and Zagier's work on singular moduli
from the perspective of Kudla's  program \cite{KuAnnals, KuMSRI, KRYBook} to relate
arithmetic intersection multiplicities on   Shimura varieties of orthogonal and unitary type to the
Fourier coefficients of derivatives of Eisenstein series.
On the occasion of his sixtieth birthday the authors wish to express  to Steve Kudla both  their  deepest appreciation
 for his beautiful mathematics, and  their gratitude for his  influence on their own lives and work.


\section{Moduli spaces of CM elliptic curves} \label{sect2}


In this section we study the moduli stack $\mathcal{X}_\alpha$ and prove Theorem
\ref{geometric theorem}.


\subsection{CM pairs}


Let $S$ be a scheme  and $R$ an order in a quadratic imaginary field.   An \emph{elliptic curve over $S$ with
complex multiplication  by $R$} is a pair $\mathbf{E}=(E,\kappa)$ in which $E\map{}S$ is an elliptic curve and
$\kappa:R\map{}\End(E)$ is an action of $R$ on $E$.

\begin{Def}\label{Def 2.1.1}
 A \emph{CM pair} over a scheme $S$ is a pair $(\mathbf{E}_1,\mathbf{E}_2)$ in which
 $\mathbf{E}_1$ and $\mathbf{E}_2$ are elliptic curves over $S$  with complex multiplication by
 $\co_{K_1}$ and $\co_{K_2}$, respectively.   An \emph{isomorphism}  between CM pairs
 $(\mathbf{E}'_1,\mathbf{E}'_2) \map{} (\mathbf{E}_1,\mathbf{E}_2)$ is a pair $(f_1,f_2)$ in which each
 $f_i:E'_i\map{}E_i$ is an $\co_{K_i}$-linear isomorphism of elliptic curves.
\end{Def}

For every CM pair  $(\mathbf{E}_1,\mathbf{E}_2)$ over a scheme $S$ we abbreviate
$$
L(\mathbf{E}_1,\mathbf{E}_2)  = \Hom(E_1,E_2),
$$
where  $\Hom$ means homomorphisms between elliptic curves over $S$ in the usual sense, and set
$$
V(\mathbf{E}_1,\mathbf{E}_2) = L(\mathbf{E}_1,\mathbf{E}_2) \otimes_\Z\Q.
$$
Assuming that $S$ is connected,  the finite rank $\Z$-module $L(\mathbf{E}_1,\mathbf{E}_2)$
is equipped with the positive definite quadratic form $\deg(j)$. Denote by
\begin{align*}
[j_1,j_2] &=  \deg(j_1+j_2) - \deg(j_1) - \deg(j_2)   \\
&=  j_1^\vee \circ j_2 +j_2^\vee \circ j_1
\end{align*}
the associated bilinear form.   The maximal order $\co_K=\co_{K_1} \otimes_\Z \co_{K_2}$
acts on the $\Z$-module  $L(\mathbf{E}_1,\mathbf{E}_2)$  by
\begin{equation}\label{K action}
(x_1\otimes x_2) \action j = \kappa_2(x_2) \circ j \circ \kappa_1( \overline{x}_1).
\end{equation}

By a  \emph{$K$-Hermitian} form on $V(\mathbf{E}_1,\mathbf{E}_2)$ we mean  a function
$$
\langle \cdot, \cdot \rangle  : V(\mathbf{E}_1,\mathbf{E}_2)  \times V(\mathbf{E}_1,\mathbf{E}_2) \map{}K
$$
which is $K$-linear in the first variable and satisfies
$\langle  j_1,j_2\rangle =  \overline{ \langle j_2,j_1\rangle }$.

\begin{Prop}\label{Prop 2.1.3}\
\begin{enumerate}
\item
There is a unique $F$-bilinear form $[j_1,j_2]_\CM$ on $V(\mathbf{E}_1,\mathbf{E}_2)$ satisfying
$$
[j_1,j_2] = \mathrm{Tr}_{F/\Q} \ [j_1,j_2]_\CM.
$$
\item
The $F$-quadratic form
$$
\cmdeg(j)=\frac{1}{2} \cdot [j,j]_\CM
$$
 is the unique $F$-quadratic form on $V(\mathbf{E}_1,\mathbf{E}_2)$ satisfying
$$
\deg(j) = \mathrm{Tr}_{F/\Q}\ \cmdeg(j).
$$
\item
There is a unique $K$-Hermitian form $\langle j_1,j_2\rangle_\CM$ on
$V(\mathbf{E}_1,\mathbf{E}_2)$  satisfying
$$
[j_1,j_2]_\CM = \mathrm{Tr}_{K/F} \ \langle j_1,j_2\rangle_\CM.
$$
\end{enumerate}
\end{Prop}

\begin{proof}
Suppose $j_1,j_2\in V(\mathbf{E}_1,\mathbf{E}_2)$.
If $x=x_1\otimes x_2\in K$ is nonzero then as elements of $\End(E_1)$ we have
\begin{eqnarray*}
[ x \action j_1 ,  j_2]
& = &
 \kappa_1(x_1)^{-1} \circ [ x \action j_1 ,  j_2] \circ  \kappa_1(x_1)   \\
& = &
 j_1^\vee \circ \kappa_2(\overline{x}_2)  \circ j_2 \circ  \kappa_1(x_1)
 + \kappa_1(\overline{x}_1) \circ j_2^\vee \circ \kappa_2( x_2) \circ j_1  \\
& = &[ j_1, \overline{x} \action j_2 ].
\end{eqnarray*}
Thus for all $x\in K$ we have $[ x \action j_1 ,  j_2]  = [ j_1, \overline{x} \action j_2 ]$.
All of the claims now follow from this property and some elementary linear algebra; in particular from the
fact that if $M/L$ is a finite separable extension of fields, then for any finite dimensional $M$-vector space
$V$ the trace $\mathrm{Tr}_{M/L}$ induces an isomorphism  $\Hom_M(V,M)\map{}\Hom_L(V,L).$
\end{proof}

Thus the complex multiplication  of $\mathbf{E}_1$ and $\mathbf{E}_2$ endows the set
$V(\mathbf{E}_1,\mathbf{E}_2)$ not only with a $K$-action, but with an $F$-quadratic form $\cmdeg$
which refines the usual notion of degree.

To understand the moduli space of CM pairs over schemes we use the
language of stacks and  groupoids as in \cite{vistoli}. Given a CM
pair $(\mathbf{E}_1,\mathbf{E}_2)$ over a scheme $S$ and a
morphism of schemes $T\map{}S$ there is an evident notion of the
pullback CM pair $(\mathbf{E}_1,\mathbf{E}_2)_{/T}$.

\begin{Def}
Define  $\mathcal{X}$ to be the category   whose objects are CM pairs over
schemes.   In the category $\mathcal{X}$ an arrow
$(\mathbf{E}'_1,\mathbf{E}'_2) \map{} (\mathbf{E}_1,\mathbf{E}_2)$
between CM pairs defined over schemes $T$ and $S$, respectively,
is a morphism of schemes $T\map{}S$  together with an isomorphism
(in the sense of Definition \ref{Def 2.1.1}) of CM pairs over $T$
$$
(\mathbf{E}'_1,\mathbf{E}'_2) \iso (\mathbf{E}_1,\mathbf{E}_2)_{/T}.
$$
\end{Def}

Thus $\mathcal{X}$ is a category fibered in groupoids over the category of schemes.
For a scheme $S$ the fiber $\mathcal{X}(S)$ is  the category of CM pairs over $S$,
and arrows in this category are isomorphisms in the sense of Definition \ref{Def 2.1.1}.

\begin{Def}
For every  $m\in \Q$ define $\mathcal{T}_m$ to be the category,
fibered in groupoids over schemes,  of triples $(\mathbf{E}_1,\mathbf{E}_2,j)$ in which
$(\mathbf{E}_1,\mathbf{E}_2)$ is a CM pair over a scheme $S$ and $j\in L(\mathbf{E}_1,\mathbf{E}_2)$
satisfies $\deg(j)=m$ on every connected component of $S$.
\end{Def}

\begin{Def}
For every  $\alpha\in F$ define
$\mathcal{X}_\alpha$ to be the category of triples $(\mathbf{E}_1,\mathbf{E}_2,j)$ in which
$(\mathbf{E}_1,\mathbf{E}_2)$ is a CM pair over a scheme and $j\in L(\mathbf{E}_1,\mathbf{E}_2)$
satisfies $\cmdeg(j)=\alpha$ on every connected component of $S$.
\end{Def}

The categories
$\mathcal{X}$, $\mathcal{T}_m$, and $\mathcal{X}_\alpha$ are algebraic  stacks  in the  sense of \cite{vistoli}
(also known as Deligne-Mumford stacks)  of finite type over $\Spec(\Z)$.  Briefly, one knows  that the category
$\mathcal{E}$ of elliptic curves over schemes is an algebraic stack of finite type over $\Spec(\Z)$,
and the relative representability of each of $\mathcal{X}$, $\mathcal{T}_m$, and $\mathcal{X}_\alpha$ over
$\mathcal{E}\times_\Z\mathcal{E}$ is proved using the methods and results of \cite[Chapter 6]{hida04}.
 For every $m\in \Q$ there is an evident  decomposition
\begin{equation}\label{moduli decomp}
\mathcal{T}_m = \bigsqcup _{\substack{  \alpha\in F  \\ \mathrm{Tr}_{F/\Q}(\alpha) = m  } } \mathcal{X}_\alpha .
\end{equation}

If   $S$ is a  scheme  and $\mathcal{C}$ is any one of $\mathcal{X}$, $\mathcal{T}_m$, or
$\mathcal{X}_\alpha$  then  $[\mathcal{C}(S)]$ denotes the set of isomorphism classes of objects in the
category $\mathcal{C}(S)$.


\subsection{The support of $\mathcal{X}_\alpha$}


Given $\alpha\in F^\times$ define a  nondegenerate $\Q$-quadratic form  $Q_\alpha$ on $K$ by
$$
Q_\alpha(x)=\mathrm{Tr}_{F/\Q}(\alpha x \overline{x}).
$$
For each place $\ell\le \infty$ of $\Q$ let $\mathrm{hasse}_\ell(\cdot)$ be the Hasse invariant  on
$\Q_\ell$-quadratic spaces (the invariant $\epsilon$ of \cite[Chapter IV.2]{serre73}) and let $(\cdot,\cdot)_\ell$ be the usual
Hilbert symbol $\Q_\ell^\times \times\Q_\ell^\times \map{}\{\pm 1\}$.  Define the \emph{local invariant}  $\inv_\ell(\alpha)=\pm 1$ by
$$
\inv_\ell(\alpha) = \mathrm{hasse}_\ell(K_\ell, Q_\alpha)\cdot  (-1,-1)_\ell.
$$
Define  the   \emph{modified local invariant} of $\alpha$ by
$$
\inv^*_\ell (\alpha)=\left\{\begin{array}{rl}
\inv_\ell (\alpha)&\mathrm{if\ } \ell <\infty \\
-\inv_\ell (\alpha)&\mathrm{if\ } \ell =\infty.
\end{array}\right.
$$
Define the \emph{support} of $\alpha$ to be the finite set of rational primes
$$
\Sppt(\alpha)=\{ \ell \le \infty : \inv^*_\ell(\alpha)=-1 \},
$$
and note that  the product formula $\prod_{\ell \le \infty}\inv_\ell(\alpha)=1$ implies that
$\Sppt(\alpha)$ has odd cardinality.

 A CM pair $(\mathbf{E}_1,\mathbf{E}_2)$ over an algebraically closed  field $k$ of nonzero characteristic
 is  \emph{supersingular} if the  underlying elliptic curves  $E_1$ and $E_2$ are supersingular in
 the usual sense.  If this is the case then $E_1$ and $E_2$ are isogenous, and there is an isomorphism of
 $\Q$-vector spaces $V(\mathbf{E}_1,\mathbf{E}_2) \iso \End(E)\otimes_\Z\Q$ for any supersingular
 elliptic curve $E$ over $k$.  In particular $V(\mathbf{E}_1,\mathbf{E}_2)$ has dimension four
 as a $\Q$-vector space, and so has dimension  one over $K$.  Hermitian $K$-modules of
 dimension one are easy to classify, and we find that  for some
$\beta\in F^\times$ there is an isomorphism of $F$-quadratic spaces
\begin{equation}\label{basic hermitian}
 ( V(\mathbf{E}_1,\mathbf{E}_2), \cmdeg ) \iso (K, \beta \cdot \mathrm{Nm}_{K/F} ).
\end{equation}
The element $\beta$ is determined up to multiplication by a norm from $K^\times$, and as the $\Q$-quadratic
form $\deg$ is positive definite,  $\beta$ is totally positive.   In the next subsection we will compute the image of
$\beta$ in $\widehat{F}^\times/\mathrm{Nm}_{K/F}(\widehat{K}^\times)$,  which is enough to
determine the isomorphism class of the quadratic space (\ref{basic hermitian}); see
 Theorem \ref{Thm:quadratic model} below.

\begin{Prop}\label{Prop 2.2.2}
Suppose $k$ is an algebraically closed field of characteristic $p\ge 0$ and $\alpha\in F^\times$.  If
$(\mathbf{E}_1,\mathbf{E}_2,j) \in \mathcal{X}_\alpha(k)$ then
\begin{enumerate}
\item
$p>0$ and the CM pair  $(\mathbf{E}_1,\mathbf{E}_2)$ is supersingular;
\item
$p$ is nonsplit in both $K_1$ and $K_2$;
\item
there are isomorphisms of  quadratic spaces over $\Q$
$$
(K,Q_\alpha) \iso (V(\mathbf{E}_1,\mathbf{E}_2), \deg) \iso (H,\mathrm{Nm})
$$
where $H$ is the rational quaternion algebra over $\Q$ of discriminant $p$ and
$\mathrm{Nm}$ is the reduced norm on $H$;
\item
  $\Sppt(\alpha)=\{p\}$ and $\alpha$ is totally positive.
\end{enumerate}
\end{Prop}

\begin{proof}
Suppose that  $p=0$.  As $j:E_1\map{}E_2$ is a nonzero isogeny,
$\kappa_1$, $\kappa_2$, and $j$ determine isomorphisms
\begin{equation}\label{not ordinary}
K_1\iso \End(E_1)\otimes_\Z\Q \iso \End(E_2)\otimes_\Z\Q\iso K_2,
\end{equation}
contrary to our hypotheses on $K_1$ and $K_2$.  Thus $k$ has  characteristic $p>0$.
The existence of the isogeny $j$ implies that the
 elliptic curves $E_1$ and $E_2$ are  either both supersingular or both ordinary.
 If they are both ordinary then again (\ref{not ordinary}) gives a contradiction.
 Thus $(\mathbf{E}_1,\mathbf{E}_2)$ is supersingular.  In particular
$$
\End(E_1) \otimes_\Z\Q \iso H \iso \End(E_2)\otimes_\Z\Q
$$
as $\Q$-algebras, and as $K_1$ and $K_2$ both embed into $H$ we see that $p$ is nonsplit
in both $K_1$ and $K_2$.   By hypothesis the quadratic space
(\ref{basic hermitian})  represents $\alpha$,  and so there is  a  $u\in K^\times$ such that
$\alpha=\beta\cdot\mathrm{Nm}_{K/F}(u)$.  Thus we  have
an isomorphism  of  $F$-quadratic spaces
$$
 ( V(\mathbf{E}_1,\mathbf{E}_2), \cmdeg ) \iso (K, \alpha \cdot \mathrm{Nm}_{K/F} )
$$
and so also an isomorphism of $\Q$-quadratic spaces
$$
 ( V(\mathbf{E}_1,\mathbf{E}_2), \deg ) \iso (K, Q_\alpha  ).
$$
Fix an isomorphism of $\Q$-algebras $H\iso \End(E_1)\otimes_\Z\Q$.  The function
$f\mapsto f^\vee \circ j$ defines an isomorphism of $\Q$-quadratic spaces
$$
(V(\mathbf{E}_1,\mathbf{E}_2), \deg) \iso (H,  b^{-1} \cdot \mathrm{Nm})
$$
where $b=\deg(j)$, and as $b$ lies in the image of the reduced norm
$H^\times\map{}\Q^\times$ there is an isomorphism of $\Q$-quadratic spaces
$(H,  b^{-1} \cdot \mathrm{Nm}) \iso (H,  \mathrm{Nm}) $.
It only remains to prove that $\Sppt(\alpha)=\{p\}$.  Using the isomorphism
$(K,Q_\alpha)\iso (H,\mathrm{Nm})$ already proved we find
$$
\inv_\ell(\alpha) = \mathrm{hasse}_\ell( H_\ell, \mathrm{Nm} )\cdot  (-1,-1)_\ell.
$$
By direct calculation of the Hasse invariant of $(H_\ell ,\mathrm{Nm})$ it follows that
$$
\inv_\ell(\alpha) = \begin{cases}
-1 &\mathrm{if\ }\ell \in \{ p,\infty \} \\
1 &\mathrm{otherwise}.
\end{cases}
$$
In particular $\inv^*_\ell(\alpha)=-1$ if and only if $\ell=p$.
\end{proof}

\begin{Cor}\label{Cor 2.2.4}
Suppose $\alpha\in F^\times$.  If $\mathcal{X}_\alpha\not=\emptyset$ then $\alpha$
is totally positive and    $\Sppt(\alpha)=\{ p \}$  for a finite  prime $p$.  Furthermore
 all geometric points of  $\mathcal{X}_\alpha$ lie in characteristic $p$ and are supersingular.
\end{Cor}

\begin{proof}
This is immediate from Proposition \ref{Prop 2.2.2}.
\end{proof}


\subsection{Local quadratic spaces}
\label{ss:quadratic}


Fix a prime $p$ that is nonsplit in both $K_1$ and $K_2$, and a supersingular CM pair
$(\mathbf{E}_1,\mathbf{E}_2)$ over $\F_p^\alg$.   For $i\in \{1,2\}$ the action
$$
\kappa^\mathrm{Lie}_i:\co_{K_i} \map{}\End_{\F_p^\alg}(\mathrm{Lie}(E_i)) \iso \F_p^\alg
$$
induces a homomorphism $\co_K\iso \co_{K_1}\otimes_\Z\co_{K_2} \map{}\F_p^\alg$ defined by
\begin{equation}\label{reflex map}
t_1\otimes t_2 \mapsto  \kappa^\mathrm{Lie}_1(t_1)\cdot \kappa^\mathrm{Lie}_2(t_2).
\end{equation}
The kernel of this map is denoted $\mathfrak{q}$.

\begin{Def}\label{Def:reflex prime}
The prime $\mathfrak{p}$ of $F$  below $\mathfrak{q}$ is  the \emph{reflex prime} of $(\mathbf{E}_1,\mathbf{E}_2)$.
\end{Def}

\begin{Rem}\label{Rem:prime description}
If $p$ is inert in both $K_1$ and $K_2$ then $p$ is split in $F$, and each prime of $F$ above $p$
is inert in $K$;  in particular there are two possibilities for the reflex prime of a
supersingular CM pair $(\mathbf{E}_1,\mathbf{E}_2)$.  If instead $p$ is ramified in one of $K_1$ or
$K_2$ then $p$ is ramified in $F$, and the unique prime of $F$ above $p$ is inert in $K$; in particular
there is only one possibility for the reflex prime of $(\mathbf{E}_1,\mathbf{E}_2)$.
\end{Rem}

For every prime $\ell$  abbreviate
\begin{eqnarray*}
L_\ell(\mathbf{E}_1,\mathbf{E}_2) &=&  L(\mathbf{E}_1,\mathbf{E}_2) \otimes_\Z\Z_\ell \\
V_\ell(\mathbf{E}_1,\mathbf{E}_2) &=&  V(\mathbf{E}_1,\mathbf{E}_2) \otimes_\Q\Q_\ell.
\end{eqnarray*}

\begin{Lem}\label{Lem:CM model}
Suppose   $\ell$  is a  prime different from $p$.   For some  $\beta \in F_{\ell}^\times$ satisfying
$\beta\co_{F,\ell} =  \mathfrak{D}^{-1}_\ell$ there is a $K_\ell$-linear  isomorphism of $F_\ell$-quadratic spaces
$$
(V_\ell(\mathbf{E}_1,\mathbf{E}_2) , \cmdeg) \iso ( K_\ell, \beta \cdot \mathrm{Nm}_{K_{\ell}/F_{\ell}} )
$$
taking $L_\ell(\mathbf{E}_1,\mathbf{E}_2)$ isomorphically to $\co_{K,\ell}$.
\end{Lem}

\begin{proof}
The existence of the desired isomorphism for some choice of $\beta \in F_{\ell}^\times$ is clear from
(\ref{basic hermitian}).  We must determine the fractional $\co_{F,\ell}$-ideal $\beta\co_{F,\ell}$.
Any choice of $\Z_\ell$-bases for the Tate modules $\mathrm{Ta}_\ell(E_1)$ and
$\mathrm{Ta}_\ell(E_2)$  determines   isomorphisms  of $\Z_\ell$-modules
$$
L_\ell(\mathbf{E}_1,\mathbf{E}_2)  \iso
\Hom_{\Z_\ell}(\mathrm{Ta}_\ell(E_1),\mathrm{Ta}_\ell(E_2)) \iso M_2(\Z_\ell),
$$
which identify the quadratic form $\deg$ with the quadratic form $u\cdot \det$ for some $u\in\Z_\ell^\times$.
It follows that the $\Z_\ell$-lattice  $L_\ell(\mathbf{E}_1,\mathbf{E}_2)$ is self dual relative to
$\deg$, and  hence that  the  $\co_{F,\ell}$-bilinear form of Proposition \ref{Prop 2.1.3}
$$
[\cdot,\cdot]_\CM : L_\ell(\mathbf{E}_1,\mathbf{E}_2) \times L_\ell(\mathbf{E}_1,\mathbf{E}_2)
 \map{}\mathfrak{D}_\ell^{-1}
$$
is  a perfect pairing.  This implies that the  $\co_{F,\ell}$-bilinear form
$$
\co_{K,\ell}\times\co_{K,\ell}\map{}\mathfrak{D}_\ell^{-1}
$$
defined by $\beta \cdot  \mathrm{Tr}_{K_\ell/F_{\ell}}(x\overline{y})$  is a perfect pairing.   As $K/F$ is
unramified, it follows   that  $\beta$ generates $\mathfrak{D}_\ell^{-1}$.
\end{proof}

 \begin{Lem}\label{Lem:CM model I}
 For some   $\beta\in F_{p}^\times $ satisfying $\beta\co_{F,p}=\mathfrak{p}\mathfrak{D}^{-1}_p$ there is a
 $K_p$-linear  isomorphism of  $F_p$-quadratic spaces
 $$
 (V_p(\mathbf{E}_1,\mathbf{E}_2) , \cmdeg) \iso (K_p, \beta\cdot \mathrm{Nm}_{K_p/F_p} )
 $$
 taking  $L_p(\mathbf{E}_1,\mathbf{E}_2)$ isomorphically to $\co_{K,p}$. Here $\mathfrak p$ is the reflex prime of $(\mathbf E_1, \mathbf E_2)$.
\end{Lem}

\begin{proof}
The existence of the desired isomorphism for some choice of $\beta \in F_p^\times$ is clear from
(\ref{basic hermitian}).  We must determine the fractional $\co_{F, p}$-ideal $\beta\co_{F,p}$.
The proof is similar to the proof of Lemma  \ref{Lem:CM model}, but with $\ell$-adic Tate modules
replaced by (covariant) Dieudonn\'e modules.   Fix any supersingular elliptic curve $E$ over $\F_p^\alg$.
By choosing prime to $p$-isogenies $E_i\map{}E$ we may reduce to the case where the CM
elliptic curves $\mathbf{E}_1$ and $\mathbf{E}_2$ have the same underlying elliptic curve $E$.
Let $D$ be the   Dieudonn\'e module of  $E$ and set $\Delta=\End(D)$, the maximal order
in a quaternion division algebra over $\Q_p$.   In this way we obtain an isomorphism of $\Z_p$-quadratic
spaces
 $$
\big( L_p(\mathbf{E}_1,\mathbf{E}_2) ,\deg \big)  \iso ( \Delta ,\mathrm{Nm})
$$
where $\mathrm{Nm}$ is the  reduced norm  on $\Delta$.   Denoting by $\mathfrak{m}_\Delta$ the
unique maximal ideal of $\Delta$,  the dual lattice of $\Delta$ relative to
$\mathrm{Nm}$ is   $\mathfrak{m}_\Delta^{-1}$.  The dual lattice of
$L_p(\mathbf{E}_1,\mathbf{E}_2)$ with respect to  $\deg$ is
$$
L_p(\mathbf{E}_1,\mathbf{E}_2)^\vee =
\big\{ j \in V_p(\mathbf{E}_1,\mathbf{E}_2)   :
[j,j']_\CM \in \mathfrak{D}_p^{-1} \quad \forall j' \in  L_p(\mathbf{E}_1,\mathbf{E}_2)  \big\},
$$
and  we have $\co_{K,p}$-linear isomorphisms
\begin{equation}\label{dual indices}
\beta^{-1} \mathfrak{D}^{-1}\co_{K,p}/\co_{K,p} \iso
L_p(\mathbf{E}_1,\mathbf{E}_2)^\vee/ L_p(\mathbf{E}_1,\mathbf{E}_2) \iso
\mathfrak{m}^{-1}_\Delta/\Delta.
\end{equation}
As a group $\Delta/\mathfrak{m}_\Delta\iso \F_{p^2}$, and so $[\co_{K,p}: \beta\mathfrak{D} \co_{K,p}] = p^2$.

Suppose first that $p$ is ramified in one of $K_1$ or $K_2$, so that  $p\co_F=\mathfrak{p}^2$.
Using Remark \ref{Rem:prime description}  and the equality
$[\co_{K,p}: \beta\mathfrak{D}\co_{K,p}] = p^2$
we immediately deduce that $\beta\mathfrak{D}_p=\mathfrak{p}\co_{F,p}$, and we are done.

Now suppose  that $p$ is unramified in $K_1$ and $K_2$, and  recall that
$\co_K$ acts on $\Delta$ by
$$
(t_1\otimes t_2) \action j = \kappa_2(t_2)  \circ  j  \circ  \kappa_1(\overline{t}_1) .
$$
If we fix a uniformizing parameter   $\Pi\in \Delta$ in such a way that    $\kappa_1(\overline{t}_1) \Pi = \Pi \kappa_1(t_1)$
for every $t_1\in \co_{K_1}$, then for any  $u\in \kappa_1(\co_{K_1})\subset \Delta$ we have
$$
(t_1\otimes t_2)\action u \Pi^{-1}  = \kappa_2(t_2) u \Pi^{-1} \kappa_1(\overline{t}_1) =
\kappa_2(t_2)\kappa_1(t_1)  \cdot u\Pi^{-1}.
$$
As $\mathfrak{m}^{-1}_\Delta/\Delta$ is generated by such elements $u\Pi^{-1}$, $\co_K$ acts on
$\mathfrak{m}^{-1}_\Delta/\Delta$ through left  multiplication by the composition
$$
\co_K\map{}  \Delta \map{}\Delta/\mathfrak{m}_\Delta
$$
where the first arrow is $t_1\otimes t_2 \mapsto \kappa_2(t_2)\kappa_1(t_1)$.  On the other hand, the action
$$\Delta\map{}\End_{\F_p^\alg}( \mathrm{Lie}(E)) \iso \F_p^\alg$$ determines an isomorphism
$\Delta/\mathfrak{m}_\Delta\iso \F_{p^2}$ which allows us to identify $\kappa_i^\mathrm{Lie}$ with the
composition
$$
\co_{K_i} \map{\kappa_i} \Delta \map{}\Delta/\mathfrak{m}_\Delta \map{}\F_{p^2}.
$$
The conclusion of all of this is that the action of $\co_K$ on $\mathfrak{m}^{-1}_\Delta/\Delta$ factors through the
kernel $\mathfrak{q}$ of the map (\ref{reflex map}).  Returning to (\ref{dual indices}) we find that
$$
\beta^{-1}\mathfrak{D}^{-1} \co_{K,p}/\co_{K,p} \iso \co_{K,p}/\mathfrak{q}
$$
as $\co_K$-modules,  and the relation  $\beta\co_{F,p}=\mathfrak{p}\mathfrak{D}_p^{-1}$
follows easily.
\end{proof}

\begin{Thm}\label{Thm:quadratic model}
For any finite idele $\beta\in \widehat{F}^\times$ satisfying
$\beta\widehat{\co}_F = \mathfrak{p}\mathfrak{D}^{-1}\widehat{\co}_F$ there is a $\widehat{K}$-linear isomorphism of
$\widehat{F}$-quadratic spaces
$$
\big(  \widehat{V}(\mathbf{E}_1,\mathbf{E}_2) , \cmdeg \big) \iso \big( \widehat{K} , \beta \cdot \mathrm{Nm}_{K/F} \big)
$$
taking  $\widehat{L}(\mathbf{E}_1,\mathbf{E}_2)$ isomorphically to $\widehat{\co}_K$.
\end{Thm}

\begin{proof}
Combining  Lemma \ref{Lem:CM model} with Lemma \ref{Lem:CM model I}
 shows that the claim is true for some $\beta\in \widehat{F}^\times$ satisfying
$\beta\widehat{\co}_F = \mathfrak{p}\mathfrak{D}^{-1}\widehat{\co}_F$.
The surjectivity of the norm map  $\widehat{\co}_K^\times\map{}\widehat{\co}_F^\times$ implies
that the claim is true for all such $\beta$.
\end{proof}


\subsection{Group actions} \label{sect2.3}


For any sets $Y\subset X$  the characteristic function of $Y$ is denoted $\mathbf{1}_Y$.
For $i\in\{1,2\}$  define an algebraic group over $\Q$  by
$$
T_i(A) = (K_i\otimes_\Q A)^\times
$$
for any $\Q$-algebra $A$.  Let $\nu_i:T_i\map{}\mathbb{G}_m$ be the norm
 $\nu_i(t_i)=t_i \overline{t}_i$ and define
$$
T(A) = \{(t_1,t_2) \in T_1(A)\times T_2(A) : \nu_1(t_1)=\nu_2(t_2) \}.
$$
Define an algebraic group $S$ over $\Q$ by
$$
S(A)=\{ z\in (K\otimes_\Q A)^\times : \mathrm{Nm}_{K/F}(z)=1\}.
$$
There is an evident character $\nu:T\map{}\mathbb{G}_m$ defined by the relations
$$
\nu_1(t_1)=\nu(t) = \nu_2(t_2)
$$
for  $t=(t_1,t_2)\in T(R)$,  and a  homomorphism $\eta:T\map{}S$  defined by
\begin{equation}\label{unitary map}
\eta(t ) =  \nu(t)^{-1} \cdot (t_1\otimes t_2) .
\end{equation}
Let $U\subset T(\A_f)$ be the compact open subgroup
$$
U=T(\A_f)\cap (\widehat{\co}_{K_1}^\times \times \widehat{\co}_{K_2}^\times )
$$
and let $V\subset S(\A_f)$ be the image of $U$ under $\eta:T(\A_f)\map{}S(\A_f)$.  The groups $U$ and $V$ factor
as  $U=\prod_\ell U_\ell$ and $V=\prod_\ell V_\ell$ for compact open subgroups $U_\ell \subset T(\Q_\ell)$ and
$V_\ell\subset S(\Q_\ell)$.

\begin{Prop}\label{Prop 2.3.1}
If $k$ is a  field of characteristic $0$, the ring of adeles $\A$,  or the ring of finite adeles $\A_f$ then the sequence
$$
1\map{}k^\times \map{} T(k)\map{\eta}S(k)\map{}1
$$
is exact, where $k^\times\map{}T(k)$ is the diagonal inclusion.
\end{Prop}

\begin{proof}
If  $k$ is an algebraically closed of characteristic $0$ then the claim is proved by explicit calculation
after choosing diagonalizations  $T(k)\iso (k^\times)^3$ and $S(k)\iso(k^\times)^2$,
and we leave this as an exercise for the reader.  The exactness for an arbitrary field of characteristic
$0$ is immediate from  the algebraically closed case and Hilbert's Theorem 90.
The proof for $k=\A$ is proved the same way,
using the  adelic form of Hilbert's Theorem 90 \cite[Corollary 8.1.3]{neukirch}, and the exactness for
$k=\A$ implies the exactness for $k=\A_f$.
\end{proof}

Using Proposition \ref{Prop 2.3.1} and the inclusion $\A_f^\times\subset T(\Q)U$, it follows that
the homomorphism  (\ref{unitary map}) induces an isomorphism
\begin{equation}\label{transfer}
T(\Q)\backslash T(\A_f) /U \iso S(\Q)\backslash S(\A_f)/V.
\end{equation}
For $i\in\{1,2\}$ let $\mathrm{Pic}(\co_{K_i})$ be the ideal class group of $K_i$ and set
$$
\Gamma = \mathrm{Pic}(\co_{K_1})  \times \mathrm{Pic}(\co_{K_2}).
$$
Define a homomorphism
\begin{equation}\label{torus ideals}
T(\Q)\backslash T(\A_f) /U \map{}\Gamma
\end{equation}
by sending   $(t_1,t_2)\in T(\A_f)$ to the pair of ideal classes $(\mathfrak{a}_1 , \mathfrak{a}_2 ) \in \Gamma$
determined  by $\mathfrak{a}_i \widehat{\co}_{K_i}= t_i\widehat{\co}_{K_i}$.

\begin{Prop}\label{Prop:index}
The homomorphism (\ref{torus ideals}) is an isomorphism.
\end{Prop}

\begin{proof}
If we identify $\mathrm{Pic}(\co_{K_i}) \iso K_i^\times\backslash \widehat{K}_i^\times/\widehat{\co}_{K_i}$
in the usual way then the map (\ref{torus ideals}) is identified with the map
$$
T(\Q)\backslash T(\A_f) /U \map{} \big(K_1^\times\backslash \widehat{K}_1^\times/\widehat{\co}_{K_1} \big)
\times
\big( K_2^\times\backslash \widehat{K}_2^\times/\widehat{\co}_{K_2} \big)
$$
defined by $(t_1,t_1) \mapsto (t_1, t_2)$, and the   injectivity  follows easily.

The surjectivity of (\ref{torus ideals}) is less obvious. For $i\in\{1,2\}$  fix a fractional $\co_{K_i}$-ideal
$\mathfrak{a}_i$, set $a_i=\mathrm{Nm}_{K_i/\Q}(\mathfrak{a}_i)$,
and define a quadratic form on the $\Q$-vector space $K_i$
$$
Q_i(x)= a_i \cdot \mathrm{Nm}_{K_i/\Q}(x).
$$
Let $W$ be the $\Q$-vector space  $K_1\oplus K_2$ endowed with the quadratic form
$$
Q(x_1,x_2)=Q_1(x_1)-Q_2(x_2).
$$
The claim is that $(W,Q)$ represents $0$, and by the Hasse-Minkowski theorem it suffices to prove this
everywhere locally.  As $W\otimes_\Q\R$ has signature $(2,2)$ it clearly represents $0$.
Fix a prime $\ell <\infty$.   The quadratic space $W_\ell$ has discriminant
$d_1d_2\in\Q_\ell^\times/(\Q_\ell^\times)^2$ and Hasse invariant
$$
(a_1,d_1)_\ell  \cdot  (a_2,d_2)_\ell  \cdot  (d_1,-d_2)_\ell   \cdot  ( -1,-1)_\ell.
$$
If $d_1 d_2$ is not a square in $\Q_\ell^\times$ then $W_\ell$ represents $0$ by
\cite[Chapter IV.2.2]{serre73}.   Thus we may assume that $d_1=d_2$ up to a square in
$\Q_\ell^\times$.  As $a_i$ is the norm of a fractional ideal in  $K_{i,\ell}$ we may factor
$a_i=u_i\cdot b_i$ with $b_i$ equal to the norm of some element in $K_\ell^\times$ and
$u_i\in\Z_\ell^\times$.  As we assume that $\gcd(d_1,d_2)=1$, at least one of $K_1$ and
$K_2$ is unramified at $\ell$.    Thus $u_1$ is either a norm from $K_{1,\ell}^\times$ or a norm from
$K_{2,\ell}^\times$, and so either $(u_1,d_1)_\ell=1$  or $(u_1,d_2)_\ell=1$.  But
$(u_1,d_1)_\ell=(u_1,d_2)_\ell$ as $d_1=d_2$ up to  a square.  Thus we have
$$
(a_1,d_1)_\ell = (u_1,d_1)_\ell =1.
$$
The same argument shows that $(a_2,d_2)_\ell=1$, and as $(d_1,-d_2)_\ell =1$
is obvious we find that the Hasse  invariant of $W_\ell$ is $(-1,-1)_\ell$.  Again by
\cite[Chapter IV.2.2]{serre73} the quadratic space $W_\ell$ represents $0$.
Having proved that the quadratic space $(W,Q)$ represents $0$, we deduce that there is an
$m\in\Q^\times$ that is  represented both by $Q_1$ and by $Q_2$.  Choosing
$r_i\in K_i^\times$ such that $Q_1(r_1)=m=Q_2(r_2)$ we see  that the fractional ideal
$\mathfrak{b}_i=\mathfrak{a}_i r_i$ lies in the same ideal class as $\mathfrak{a}_i$, and that
\begin{equation}\label{common norm}
\mathrm{Nm}_{K_1/\Q}(\mathfrak{b}_1) = \mathrm{Nm}_{K_2/\Q}(\mathfrak{b}_2) .
\end{equation}
Thus we have proved that every element of $\Gamma$ has the form
$( \mathfrak{b}_1  ,  \mathfrak{b}_2 )$ with  $\mathfrak{b}_1$ and $\mathfrak{b}_2$
satisfying (\ref{common norm}).  Now choose $t_i\in \widehat{K}_i^\times$ satisfying
$$
t_i \widehat{\co}_{K,i}^\times = \mathfrak{b}_i\widehat{\co}_{K,i}^\times.
$$
The relation (\ref{common norm}) implies that there is a $u\in \widehat{\Z}^\times$ such that
$$
\mathrm{Nm}_{K_1/\Q}(t_1) = u\cdot \mathrm{Nm}_{K_2/\Q}(t_2).
$$
The hypothesis  $\gcd(d_1,d_2)=1$ implies that
$$
\widehat{\Z}^\times  = \mathrm{Nm}_{K_1/\Q}(\widehat{\co}_{K_1}^\times)  \cdot
 \mathrm{Nm}_{K_2/\Q}(\widehat{\co}_{K_2}^\times).
$$
Factoring $u$ as the product of the norm of some  $v_1^{-1}\in \widehat{\co}_{K_1}^\times$ and the  norm of
some  $v_2\in \widehat{\co}_{K_2}^\times$ we may then replace $t_i$ by $t_iv_i$ so that $(t_1,t_2)\in T(\A_f)$.
This proves the surjectivity of (\ref{torus ideals}), and completes the proof.
\end{proof}

For any scheme $S$ the  group $\Gamma$ acts on the set $[\mathcal{X}(S)]$ on the right  by Serre's tensor construction
\cite[Section 7]{conrad04}
$$
  (\mathbf{E}_1,\mathbf{E}_2) \otimes  ( \mathfrak{a}_1 , \mathfrak{a}_2 )
  = (\mathbf{E}_1\otimes \mathfrak{a}_1, \mathbf{E}_2\otimes  \mathfrak{a}_2 )
$$
(the tensor products on the right are over $\co_{K_1}$ and $\co_{K_2}$, respectively).

\begin{Rem}\label{four orbits}
The classical theory of complex multiplication implies that the action of $\Gamma$ on
$[\mathcal{X}(\C)]$ breaks $[\mathcal{X}(\C)]$  into a disjoint union of four simply transitive orbits.
The orbits are indexed by the set of pairs
 $$
\{ (\pi_1,\pi_2) :  \pi_i \in \Hom_{\Q-\mathrm{alg}}( K_i , \C)  \},
 $$
and the isomorphism class of a CM pair $(\mathbf{E}_1,\mathbf{E}_2) \in \mathcal{X}(\C)$
lies in the orbit indexed by $(\pi_1,\pi_2)$  if and only if  the action of $K_i$ on the $1$-dimensional $\C$-vector
space $\mathrm{Lie}(E_i)$  is through $\pi_i$ for each $i\in\{1,2\}$.
\end{Rem}

Suppose  $(\mathbf{E}_1,\mathbf{E}_2)$ is a supersingular CM pair over an algebraically closed
field of nonzero characteristic. If follows from (\ref{basic hermitian}) and  \cite[Corollary V.6.1.3]{Knus} that the
restriction of the action (\ref{K action})   to the subgroup $S(\Q) \subset K^\times$ identifies
 $$
 S \iso \mathrm{Res}_{F/\Q} \mathrm{SO}( V(\mathbf{E}_1,\mathbf{E}_2) , \cmdeg).
 $$
 The group  $T(\Q)$ then  acts on  $V(\mathbf{E}_1,\mathbf{E}_2)$ through orthogonal transformations
 by composing with  the homomorphism $\eta:T\map{}S$, and this action is given by the simple formula
\begin{equation*}
t \action j = \kappa_2(t_2)\circ j \circ \kappa_1(t_1)^{-1}
\end{equation*}
for   $t=(t_1,t_2)\in T(\Q)$.    To understand the relation between the action of $T$ on
$V(\mathbf{E}_1,\mathbf{E}_2)$ and the action of $\Gamma$ on the set of all supersingular CM pairs,
fix a $t=(t_1,t_2) \in T(\A_f)$ and let $(\mathfrak{a}_1,\mathfrak{a}_2)$ be the image of $t$
under (\ref{torus ideals}).   For $i\in \{1,2\}$ there is an $\co_K$-linear quasi-isogeny
$$
f_i \in \Hom( \mathbf{E}_i ,  \mathbf{E}_i\otimes\mathfrak{a}_i )\otimes_\Z\Q
$$
defined by $f_i(x)=x\otimes 1$, and the $\widehat{K}$-linear isomorphism of
$\widehat{F}$-quadratic spaces
$$
V(\mathbf{E}_1,\mathbf{E}_2)
\iso
V(\mathbf{E}_1\otimes\mathfrak{a}_1,\mathbf{E}_2\otimes\mathfrak{a}_2)
$$
defined by $j\mapsto  f_2\circ j \circ f_1^{-1}$ identifies
$\widehat{L}(\mathbf{E}_1\otimes\mathfrak{a}_1,\mathbf{E}_2\otimes\mathfrak{a}_2)$
with the $\widehat{\co}_K$-submodule
$$
t\action \widehat{L}(\mathbf{E}_1,\mathbf{E}_2)
= \{ \kappa_2(t_2)\circ j\circ \kappa_1(t_1)^{-1} : j\in \widehat{L}(\mathbf{E}_1,\mathbf{E}_2)  \}
$$
of $\widehat{V}(\mathbf{E}_1,\mathbf{E}_2)$.

As above, let $(\mathbf{E}_1,\mathbf{E}_2)$ be a supersingular CM pair over an algebraically closed \
field of nonzero characteristic. Given a prime $\ell$ and an $\alpha\in F_{\ell}^\times$,  the
\emph{orbital integral} at $\ell$ is defined  by
\begin{equation}\label{orbital}
O_\ell(\alpha,\mathbf{E}_1,\mathbf{E}_2)
= \sum_{t \in \Q_\ell^\times \backslash T(\Q_\ell)/ U_\ell }
\mathbf{1}_{L_\ell(\mathbf{E}_1,\mathbf{E}_2)} (t^{-1} \action j )
\end{equation}
if there exists a $j\in V_\ell(\mathbf{E}_1,\mathbf{E}_2)$ satisfying $\cmdeg(j)=\alpha$.
If no such $j$ exists then set $O_\ell(\alpha,\mathbf{E}_1,\mathbf{E}_2) =0$.
As $T(\Q_\ell)$ acts transitively on the set of all
$j\in V_\ell(\mathbf{E}_1,\mathbf{E}_2)$ for which $\cmdeg(j)=\alpha$,
the orbital integral does not depend on the choice of $j$ used in its definition.
If $t \in T(\A_f)$ has image $(\mathfrak{a}_1,\mathfrak{a}_2) \in \Gamma$ under the isomorphism
(\ref{torus ideals}) then
\begin{eqnarray*}
O_\ell(\alpha,\mathbf{E}_1 \otimes \mathfrak{a}_1,\mathbf{E} \otimes \mathfrak{a}_2)
&=&
\sum_{ s\in \Q_\ell^\times\backslash T(\Q_\ell) /U_\ell }
\mathbf{1}_{L_\ell(\mathbf{E}_1\otimes\mathfrak{a}_1,\mathbf{E}_2\otimes\mathfrak{a}_2)}(s^{-1} \action j) \\
& = &
\sum_{ s\in \Q_\ell^\times\backslash T(\Q_\ell) /U_\ell }
 \mathbf{1}_{ t\action L_\ell(\mathbf{E}_1,\mathbf{E}_2)}(s^{-1} \action j) \\
& = & O_\ell(\alpha,\mathbf{E}_1,\mathbf{E}_2),
\end{eqnarray*}
and so  the orbital integral is constant on $\Gamma$-orbits.

\begin{Lem}\label{C stab}
Let $k$ be an algebraically closed field and recall $\mathbf{w}_i = |\co_{K_i}^\times|$.
 Every  $x\in [\mathcal{X}(k)]$ has trivial stabilizer in $\Gamma$ and satisfies
$$
 \big| \Aut_{\mathcal{X}(k)}(x)   \big| =\mathbf{w}_1\mathbf{w}_2.
$$
\end{Lem}

\begin{proof}
Suppose we have a $(\mathfrak{a}_1,\mathfrak{a}_2 ) \in \Gamma$ and a CM pair
$(\mathbf{E}_1,\mathbf{E}_2)$ defined over $k$ with the property that
$$
(\mathbf{E}_1,\mathbf{E}_2) \iso
(\mathbf{E}_1\otimes  \mathfrak{a}_1, \mathbf{E}_2\otimes  \mathfrak{a}_2 ).
$$
In particular there is an isomorphism of $\co_{K_i}$-modules
$$
\Hom_{\co_{K_i}}(E_i, E_i) \iso \Hom_{\co_{K_i}}(E_i, E_i\otimes \mathfrak{a}_i)
$$
and hence, by \cite[Lemma 7.14]{conrad04},
$$
\End_{\co_{K_i}}(E_i) \iso  \End_{\co_{K_i}}(E_i)  \otimes \mathfrak{a}_i .
$$
Both as a ring and as  an $\co_{K_i}$-module $\End_{\co_{K_i}}(E_i)  \iso \co_{K_i}$, and so
$\mathfrak{a}_i\iso \co_{K_i}$ as an $\co_{K_i}$-module.  Thus $\mathfrak{a}_i$ is a principal ideal.
The equality $|\Aut_{\mathcal{X}(k)}(x) | =\mathbf{w}_1\mathbf{w}_2$ is clear from
$\Aut_{\co_{K_i}}(E_i)\iso \co_{K_i}^\times$.
\end{proof}

\begin{Lem}\label{Lem:unit calc}
We have the equalities
$$
| S(\Q)\cap V |  =  | (T(\Q)\cap U)/\{\pm 1\} | = \frac{\mathbf{w}_1 \mathbf{w}_2}{2}.
$$
\end{Lem}

\begin{proof}
The relation $T(\Q) \cap U = \co_{K_1}^\times \times \co_{K_2}^\times$ implies the second equality
in the statement of the lemma.  By Proposition \ref{Prop 2.3.1} the kernel of
$$
T(\Q) \cap U \map{\eta} S(\Q)\cap V
$$
is $\Q^\times\cap U=\{\pm 1\}$, and so  $|S(\Q)\cap V| \ge  \mathbf{w}_1 \mathbf{w}_2/2$. On the other hand,
$$
S(\Q) \cap V \subset S(\Q) \cap \widehat{\co}_K^\times =\mu_K
$$
 (the group of roots of unity in $K$), and so
 $$
 |S(\Q) \cap V| \le |\mu_K| = \frac{\mathbf{w}_1 \mathbf{w}_2}{2}.
 $$
 \end{proof}

\begin{Prop}\label{Prop:unfolding}
Suppose we are given a totally positive $\alpha\in F^\times$, a prime $p$ nonsplit in both $K_1$ and $K_2$,
and a supersingular CM pair  $(\mathbf{E}_1,\mathbf{E}_2)$ over $\F_p^\alg$.  Then
$$
\sum_{ (\mathfrak{a}_1,\mathfrak{a}_2)  \in \Gamma}
\big| \{ j \in L(\mathbf{E}_1 \otimes\mathfrak{a}_1,\mathbf{E}_2\otimes\mathfrak{a}_2)
: \cmdeg(j) = \alpha \} \big|
= \frac{\mathbf{w}_1\mathbf{w}_2}{2 }
\prod_{\ell<\infty }  O_\ell(\alpha,\mathbf{E}_1,\mathbf{E}_2) .
$$
\end{Prop}

\begin{proof}
Using the isomorphisms (\ref{transfer}) and (\ref{torus ideals}), the left hand side of the desired equality
is equal to
\begin{eqnarray}\lefteqn{
 \sum_{ ( \mathfrak{a}_1 , \mathfrak{a}_2) \in \Gamma  }
\sum_{ \substack{ j\in V(\mathbf{E}_1\otimes \mathfrak{a}_1,\mathbf{E}_2\otimes \mathfrak{a}_2) \\
\cmdeg(j)=\alpha } }
 \mathbf{  1 }_{L(\mathbf{E}_1\otimes \mathfrak{a}_1,\mathbf{E}_2\otimes \mathfrak{a}_2)} (j)
  \label{orbital chain I} } \\
 & = &
  \sum_{ t \in T(\Q)\backslash T(\A_f)/U  }
\sum_{ \substack{ j\in V(\mathbf{E}_1,\mathbf{E}_2) \\ \cmdeg(j)=\alpha } }
 \mathbf{  1 }_{t\action \widehat{L}(\mathbf{E}_1 ,\mathbf{E}_2)} (j) \nonumber  \\
 & = &
  \sum_{ s \in S(\Q)\backslash S(\A_f)/V  }
\sum_{ \substack{ j\in V(\mathbf{E}_1,\mathbf{E}_2) \\ \cmdeg(j)=\alpha } }
 \mathbf{  1 }_{s\action \widehat{L}(\mathbf{E}_1 ,\mathbf{E}_2)} (j).\nonumber
 \end{eqnarray}
Let us assume  there is some  $j_0\in  V(\mathbf{E}_1,\mathbf{E}_2)$ for which
$\cmdeg(j_0)=\alpha$.  As the group $S(\Q)$ acts simply transitively on the set of
all such $j_0$,  the final expression in (\ref{orbital chain I})  may be rewritten as
\begin{eqnarray}\lefteqn{
\sum_{ s \in S(\Q)\backslash S(\A_f)/V  }
\sum_{ \substack{ j\in V(\mathbf{E}_1,\mathbf{E}_2) \\ \cmdeg(j)=\alpha } }
 \mathbf{  1 }_{s\action \widehat{L}(\mathbf{E}_1 ,\mathbf{E}_2)} (j)  \nonumber } \\
&=&
  \sum_{ s \in S(\Q)\backslash S(\A_f)/V  }
\sum_{ \gamma \in  S(\Q) }
 \mathbf{  1 }_{s\action \widehat{L}(\mathbf{E}_1 ,\mathbf{E}_2)} (\gamma^{-1} j_0) \nonumber \\
   & = &
|S(\Q)\cap V|  \sum_{ s \in  S(\A_f)/V  }
 \mathbf{  1 }_{s \action \widehat{L}(\mathbf{E}_1 ,\mathbf{E}_2)} ( j_0)  \nonumber   \\
  & = &
\frac{\mathbf{w}_1 \mathbf{w}_2}{2} \prod_{\ell} O_\ell(\alpha,\mathbf{E}_1,\mathbf{E}_2).  \label{orbital chain II}
  \end{eqnarray}
In the final equality we have used Proposition \ref{Prop 2.3.1} and Lemma  \ref{Lem:unit calc}.
If no such $j_0$ exists then (by the Hasse-Minkowski Theorem)
both the first and last expression in (\ref{orbital chain II}) vanish.
\end{proof}


\subsection{Calculation of orbital integrals}
\label{ss:orbital}


Fix a prime $p$ that is nonsplit in both $K_1$ and $K_2$, and a supersingular CM pair
$(\mathbf{E}_1,\mathbf{E}_2)$ over $\F_p^\alg$.   In this subsection we will evaluate the local
orbital integral of (\ref{orbital}) for every prime $\ell$, and so
obtain an explicit formula for the left hand side of the equality of Proposition \ref{Prop:unfolding}.

Let $\mathfrak{p}\subset\co_F$ be the   reflex prime of $(\mathbf{E}_1,\mathbf{E}_2)$ in
the sense of Definition \ref{Def:reflex prime}.  Fix a finite idele $\beta\in \widehat{F}^\times$ such that
$\beta\widehat{\co}_F=\mathfrak{p}\mathfrak{D}^{-1}\widehat{\co}_F$.  As in the introduction,
 for any prime $\ell$ and any fractional $\co_{F,\ell}$-ideal $\mathfrak{b}$ we define
$\rho_\ell(\mathfrak{b})$ to be the number of ideals  $\mathfrak{B}\subset \co_{K,\ell}$
for which $\mathrm{Nm}_{K/F}(\mathfrak{B}) = \mathfrak{b}$.

\begin{Lem}\label{Lem:orbital eval ell}
If  $\alpha\in F_\ell^\times$ and $\ell\not=p$ then
$$
O_\ell(\alpha,\mathbf{E}_1,\mathbf{E}_2)  = \rho_\ell(\alpha\mathfrak{D}_\ell).
$$
\end{Lem}

\begin{proof}
Fix an isomorphism
$$
(V_\ell(\mathbf{E}_1,\mathbf{E}_2) , \cmdeg) \iso (K_\ell, \beta_\ell \cdot \mathrm{Nm}_{K_{\ell}/F_{\ell}} )
$$
as in Lemma \ref{Lem:CM model}.  Proposition \ref{Prop 2.3.1}  implies that
$(t_1,t_2)\mapsto \nu(t)^{-1} (t_1\otimes t_2)$ defines an isomorphism
\begin{equation}\label{orbital cosets}
\Q_\ell^\times \backslash T(\Q_\ell)/U_\ell \map{} S(\Q_\ell)/V_\ell
\end{equation}
which allows us to rewrite the orbital integral (\ref{orbital})  as
\begin{equation}\label{ell re-orbital}
O_\ell(\alpha,\mathbf{E}_1,\mathbf{E}_2) =
\sum_{s \in  S(\Q_\ell)/ V_\ell } \mathbf{1}_{\co_{K,\ell}} (s^{-1}  \phi )
\end{equation}
where $\phi\in K_\ell$ satisfies $\mathrm{Nm}_{K_\ell/F_{\ell}}(\phi) = \alpha \beta^{-1}_\ell$.
If no such $\phi$ exists then of course the orbital integral vanishes.

Suppose first that $\ell$ is inert in both $K_1$ and $K_2$, so that
$$
\co_{K,\ell}\iso \Z_{\ell^2}\times\Z_{\ell^2} \hspace{1cm} \co_{F,\ell} \iso \Z_\ell\times\Z_\ell.
$$
In this case $\Q_\ell^\times\backslash T(\Q_\ell)/U_\ell = \{1\}$ and  (\ref{ell re-orbital}) shows that
$O_\ell(\alpha,\mathbf{E}_1,\mathbf{E}_2)=1$ if there is a $\phi\in K_\ell$ satisfying
$\mathrm{Nm}_{K_\ell/F_{\ell}}(\phi) = \alpha \beta^{-1}_\ell$.
Otherwise $O_\ell(\alpha,\mathbf{E}_1,\mathbf{E}_2) =0$.
It follows that   $O_\ell(\alpha,\mathbf{E}_1,\mathbf{E}_2) = \rho_\ell(\alpha\mathfrak{D}_\ell)$
as both sides are equal to  $1$ if  $\ord_w(\alpha\beta^{-1}_\ell)$ is even and nonnegative for both
places $w$ of $F$ above $\ell$,  and otherwise both sides are zero.

Suppose next that $\ell$ is inert in $K_1$ and is ramified in $K_2$.  Then  $F_{\ell}/\Q_\ell$ is a
ramified field extension and $K_\ell/F_{\ell}$ is an unramified field extension. Again one has
$\Q_\ell^\times\backslash T(\Q_\ell)/U_\ell = \{1\}$ and  (\ref{ell re-orbital}) shows that
$O_\ell(\alpha,\mathbf{E}_1,\mathbf{E}_2)=1$ if there is a $\phi\in K_\ell$ satisfying
$\mathrm{Nm}_{K_\ell/F_{\ell}}(\phi) = \alpha \beta^{-1}_\ell$.
Otherwise $O_\ell(\alpha,\mathbf{E}_1,\mathbf{E}_2) =0$.    It follows that
$O_\ell(\alpha,\mathbf{E}_1,\mathbf{E}_2) = \rho_\ell(\alpha\mathfrak{D}_\ell)$, as both sides are
equal to $1$ if  $\ord_w(\alpha\beta^{-1}_\ell)$ is even and nonnegative for the unique
place $w$ of $F$ above $\ell$,  and otherwise both sides are zero.  The case of
$\ell$ ramified in $K_1$ and inert in $K_2$ is identical.

Suppose next that $\ell$ is split in $K_1$ and nonsplit in $K_2$.  Fix an isomorphism
$\co_{K_1,\ell}\iso \Z_\ell\times\Z_\ell$ and a uniformizer $\varpi\in \co_{K_{2,\ell}}$.
Let $\sigma$ be the nontrivial Galois automorphism of $K_{2,\ell}$  and define
$$
t_1=(1,\mathrm{Nm}_{K_{2,\ell}/\Q_\ell}(\varpi) ) \in K_{1,\ell}^\times
\hspace{1cm}
t_2=\varpi^\sigma  \in K_{2,\ell}^\times .
$$
Then $\Q_\ell^\times\backslash T(\Q_\ell) /U_\ell$ is the infinite cyclic group generated by
$t=(t_1,t_2)$.  If we  identify
$$
K_\ell \iso K_{1,\ell}\otimes_{\Q_\ell} K_{2,\ell}  \iso K_{2,\ell}  \times K_{2,\ell}
$$
using  $(x_1,x_2)\otimes y\mapsto (x_1y, x_2y^\sigma)$  then
\begin{equation}\label{simple orbital I}
F_{\ell}  \iso   \{ (a,b)\in K_{2,\ell}  \times K_{2,\ell}  : a=b\}
\end{equation}
and
$$
S(\Q_\ell) \iso \{ (a,b) \in  K_{2,\ell}^\times  \times K_{2,\ell}^\times   : ab=1\}.
$$
Using the isomorphism (\ref{orbital cosets}) and the above generator
$t\in \Q_\ell^\times\backslash T(\Q_\ell)/U_\ell$
we find that $S(\Q_\ell)/V_\ell$ is the infinite cyclic group generated by $(\varpi ,\varpi^{-1})$.
It now follows from (\ref{ell re-orbital}) that
$$
O_\ell(\alpha,\mathbf{E}_1,\mathbf{E}_2)
= \sum_{i=-\infty}^\infty \mathbf{1}_{\co_{K_2,\ell}}(\varpi^i \phi_1)\cdot
\mathbf{1}_{ \co_{K_2,\ell} }(\varpi^{-i} \phi_2)
$$
where $(\phi_1,\phi_2)\in  K_{2,\ell}^\times  \times K_{2,\ell}^\times$  is any element that satisfies
$$
(\phi_1\phi_2,\phi_1\phi_2) = \alpha\beta^{-1}_\ell
$$
under the identification (\ref{simple orbital I}).   If we let $w$ be the unique place of $F$ above
$\ell$ then  $O_\ell(\alpha,\mathbf{E}_1,\mathbf{E}_2) = \rho_\ell(\alpha\mathfrak{D}_\ell)$ as both sides
are $1+\ord_w(\alpha\mathfrak{D}_\ell)$ if $\ord_w(\alpha\mathfrak{D}_\ell)\ge 0$, and otherwise both sides are zero.
The case of $\ell$ nonsplit in $K_1$ and split in $K_2$ is identical.

Finally suppose that $\ell$ is split in both $K_1$ and $K_2$ and fix isomorphisms
$$
K_1\iso \Q_\ell\times\Q_\ell  \hspace{1cm}  K_2\iso \Q_\ell\times\Q_\ell .
$$
Define
$$
\rho_{i,j}=(\ell^i,\ell^j) \in \Q_\ell^\times\times\Q_\ell^\times.
$$
The group $\Q_\ell^\times\backslash T(\Q_\ell)/U_\ell$ is then isomorphic to the quotient of
$$
\{ (\rho_{a,b},\rho_{c,d}) \in K_{1,\ell}^\times\times K_{2,\ell}^\times : a+b=c+d \}
$$
by the subgroup $ \{ (\rho_{a,b},\rho_{c,d}) \in K_{1,\ell}^\times\times K_{2,\ell}^\times : a=b=c=d \}.$
If we  identify
\begin{equation}\label{simple orbital II}
\co_{K,\ell}\iso \co_{K_1,\ell} \otimes \co_{K_2,\ell} \iso \Z_\ell\times\Z_\ell\times\Z_\ell\times\Z_\ell
\end{equation}
via $(x_1,x_2)\otimes(y_1,y_2)\mapsto ( x_1y_1, x_2y_2, x_1y_2, x_2y_1 )$  then
$$
\co_{F , \ell} =\{ (z_1,z_2, z_3,z_4) \in  \Z_\ell\times\Z_\ell\times\Z_\ell\times\Z_\ell : z_1=z_2, z_3=z_4 \}
$$
and
$$
S(\Q_\ell)\iso \{ (z_1,z_2, z_3,z_4) \in  \Z_\ell\times\Z_\ell\times\Z_\ell\times\Z_\ell : z_1z_2 = 1,z_3z_4=1 \} .
$$
The isomorphism (\ref{orbital cosets})  takes $(\rho_{a,b},\rho_{c,d})$ to the quadruple
$( \ell^i,\ell^{-i}, \ell^j, \ell^{-j} ) \in S(\Q_\ell)$ where $i=c-b=a-d$ and $j=d-b=a-c$, and a complete
set of coset representatives for $S(\Q_\ell)/V_\ell$ is given by the set
$\{   ( \ell^i,\ell^{-i}, \ell^j, \ell^{-j} )  : i,j\in\Z\}.$  It now follows from (\ref{ell re-orbital}) that
$$
O_\ell(\alpha,\mathbf{E}_1,\mathbf{E}_2) =
\sum_{ -\infty < i,j < \infty} \mathbf{1}_{\Z_\ell}(\ell^i \phi_1) \cdot \mathbf{1}_{ \Z_\ell }(\ell^{-i} \phi_2) \cdot
\mathbf{1}_{\Z_\ell}(\ell^j \phi_3) \cdot \mathbf{1}_{ \Z_\ell }(\ell^{-j} \phi_4)
$$
where $(\phi_1,\phi_2,\phi_3,\phi_4)\in   \Z_\ell\times\Z_\ell\times\Z_\ell\times\Z_\ell\iso \co_{F,\ell} $  satisfies
$$
(\phi_1\phi_2,\phi_1\phi_2,\phi_3\phi_4,\phi_3\phi_4) = \alpha\beta^{-1}_\ell
$$
under (\ref{simple orbital II}).    If we let $w_1, w_2$ be the two places of $F$ above $\ell$ then
$O_\ell(\alpha,\mathbf{E}_1,\mathbf{E}_2) = \rho_\ell(\alpha\mathfrak{D}_\ell)$ as both sides are
  $$
  (1+\ord_{w_1}(\alpha\mathfrak{D}_\ell)) (1+\ord_{w_2}(\alpha\mathfrak{D}_\ell))
  $$
   if $\ord_{w_1}(\alpha\mathfrak{D}_\ell)\ge 0$ and  $\ord_{w_2}(\alpha\mathfrak{D}_\ell)\ge 0$,
   and otherwise both sides are zero.
\end{proof}

\begin{Lem}\label{Lem:orbital eval p}
For any $\alpha\in F_{p}^\times$
$$
O_p(\alpha,\mathbf{E}_1,\mathbf{E}_2) = \rho_p(\alpha\mathfrak{p}^{-1}\mathfrak{D}_p).
$$
\end{Lem}

\begin{proof}
As $p$ is unramified in at least one of $K_1$ or $K_2$, it is easy to see that
$$
\Q_p^\times\backslash T(\Q_p)/U_p =\{1\}.
$$
Thus  the orbital integral $O_p(\alpha,\mathbf{E}_1,\mathbf{E}_2)$ is $1$ if there is a
$j\in L_p(\mathbf{E}_1,\mathbf{E}_2)$ satisfying $\cmdeg(j)=\alpha$ and is $0$ otherwise.
Using the model of Lemma \ref{Lem:CM model I} we see that $O_p(\alpha,\mathbf{E}_1,\mathbf{E}_2)=1$
if and only if there is a $j\in \co_{K,p}$ satisfying
$$
\mathrm{Nm}_{K_p/F_{p}} (j) =\alpha\beta_p^{-1}.
$$
As each prime of $F$ above $p$ is inert in $K$,
such a $j$ exists if and only if  $\ord_w(\alpha \beta_p^{-1})$ is even and nonnegative for each
place $w$ of $F$ above $p$.  Using
$$\ord_w(\alpha\beta_p^{-1}) = \ord_w(\alpha\mathfrak{p}^{-1}\mathfrak{D}_p)$$
we find that both sides of the desired equality are $1$ if  $\ord_w(\alpha\mathfrak{p}^{-1}\mathfrak{D}_p)$
is even and nonnegative for both places $w$ of $F$ above $p$, and otherwise both sides of the
equality are $0$.
\end{proof}

Recall from the introduction that for any fractional $\co_F$-ideal $\mathfrak{b}$, $\rho(\mathfrak{b})$
is defined to be the number of ideals $\mathfrak{B}\subset\co_K$ for which
$\mathrm{Nm}_{K/F}(\mathfrak{B}) =\mathfrak{b}$.

\begin{Thm}\label{Thm:global orbital}
For any $\alpha\in F^\times$ we have
$$
\prod_{\ell} O_\ell(\alpha,\mathbf{E}_1,\mathbf{E}_2) = \rho(\alpha\mathfrak{D}\mathfrak{p}^{-1}).
$$
\end{Thm}

\begin{proof}
This is clear from Lemmas \ref{Lem:orbital eval ell} and \ref{Lem:orbital eval p}, and the
product formula (\ref{product ideals}).
\end{proof}


\subsection{Deformation theory}


Let $p$ be a prime that is nonsplit in both $K_1$ and $K_2$.  This implies that all CM pairs over
$\F_p^\alg$ are supersingular. Let $W=W(\F_p^\alg)$ be the ring of Witt vectors of $\F_p^\alg$, and let $
\Z_{p^2}\subset W$ be the  ring of Witt vectors of $\F_{p^2}$.
  Denote by  $\mathcal{CLN}$ be the category of complete local Noetherian $W$-algebras with
  residue field $\F_p^\alg$.    Let $\mathfrak{g}$ be the unique (up to isomorphism) connected
  $p$-divisible group of height $2$ and dimension $1$ over $\F_p^\alg$, and set $\Delta=\End(\mathfrak{g})$.
    Thus $\mathfrak{g}$ is isomorphic  to the $p$-divisible group of any supersingular elliptic curve over
    $\F_p^\alg$, and $\Delta$ is the maximal order in a quaternion division algebra
  over $\Q_p$.    Let $x\mapsto x^\iota$ denote the main involution on $\Delta$, let
$\mathfrak{m}_\Delta\subset\Delta$ be the maximal ideal, and let $\ord_\Delta$ be the valuation on
$\Delta$ defined by
$$
\ord_\Delta(j)=k \iff j\in \mathfrak{m}_\Delta^k\smallsetminus \mathfrak{m}_\Delta^{k+1}.
$$

  For any $\Z_p$-subalgebra $\co\subset\Delta$ denote by  $\mathrm{Def}(\mathfrak{g},\co)$
  the functor  from   $\mathcal{CLN}$  to the category of sets that
  assigns to an object $R$ of $\mathcal{CLN}$ the set of isomorphism  classes of deformations of
  $\mathfrak{g}$, with its $\co$-action, to $R$.    Suppose that $\co=\co_L$ is the maximal order in
  a quadratic field extension of $\Q_p$ and let $\pi_L\in \co_L$ be a uniformizer.
  Let $\mathcal{W}_L$ be the completion of the  ring of integers  of the
  maximal unramified extension of $L$, and choose a continuous ring homomorphism $W\map{}\mathcal{W}_L$.
  By Lubin-Tate theory (see for example \cite[Theorem 3.8]{ARGOS-7}) the deformation functor
  $\mathrm{Def}(\mathfrak{g},\co_L)$   is represented by  $\mathcal{W}_L$, so that
  $$
  \mathrm{Def}(\mathfrak{g},\co_L)(R) \iso \Hom_\mathcal{CLN}(\mathcal{W}_L,R).
  $$
  Let $\mathfrak{G}$ be the
  universal deformation of $\mathfrak{g}$, with its $\co_L$-action, to
  $\mathcal{W}_L$, and let $\mathfrak{G}_k$ be the reduction of
  $\mathfrak{G}$ to    $\mathcal{W}_L / \pi_L^k \mathcal{W}_L$.    It is a result of Gross
(see \cite[Proposition 3.3]{gross86} or  \cite[Theorem 1.4]{ARGOS-8})   that the reduction map
 $\End(\mathfrak{G}_k)\map{}\End(\mathfrak{g})$ identifies
\begin{equation}\label{gross deform}
\End(\mathfrak{G}_k) \iso \co_L+ \pi_L^{k-1}\Delta.
\end{equation}
Given any $j\in\Delta\smallsetminus \co_L$, it follows from \cite[Proposition 2.9]{rapoport96}
 that  the functor  $\mathrm{Def}(\mathfrak{g},\co_L[j])$ is represented by
 $\mathcal{W}_L/\pi_L^k\mathcal{W}_L$
 where $k$ is the largest integer such that $j$ lifts (necessarily uniquely)
 to an endomorphism of $\mathfrak{G}_k$.

  Given a CM pair $(\mathbf{E}_1,\mathbf{E}_2)$ over $\F_p^\alg$  let
$\mathrm{Def}(\mathbf{E}_1,\mathbf{E}_2)$ be the functor that assigns to every object $R$
of $\mathcal{CLN}$ the set of isomorphism classes of deformations of the pair
$(\mathbf{E}_1,\mathbf{E}_2)$ to $R$.  Fix isomorphisms
of $p$-divisible groups
$$
E_1[p^\infty] \iso \mathfrak{g} \iso E_2[p^\infty].
$$
Let $\co_{L_i} \subset \Delta$
be the image of $\kappa_i: \co_{K_i,p} \map{}\Delta$ and let $L_i\iso K_{i,p}$ be the fraction field of
$\co_{L_i}$.   It follows from the Serre-Tate theorem and the comments of the preceding paragraph that the
functor
$$
\mathrm{Def}(\mathbf{E}_1,\mathbf{E}_2) \iso \mathrm{Def}(\mathfrak{g},\co_{L_1}) \times
\mathrm{Def}(\mathfrak{g},\co_{L_2}),
$$
is represented by $\mathcal{W}_{L_1}\widehat{\otimes}_W\mathcal{W}_{L_2}$.

\begin{Prop}\label{Prop:naked count}
Let  $\mathfrak{p}$ be a  prime of $F$ above $p$.   There are $2\cdot |\Gamma|$ isomorphism
classes of CM pairs over $\F_p^\alg$ having reflex prime $\mathfrak{p}$.
\end{Prop}

\begin{proof}
Let $(\mathbf{E}_1,\mathbf{E}_2)$ be a  CM pair over $\F_p^\alg$.
Suppose first that $p$ is unramified in both $K_1$ and $K_2$, so that
$\mathrm{Def}(\mathbf{E}_1,\mathbf{E}_2)$ is represented by $W\widehat{\otimes}_W W\iso W$.
If $R$ is any object of $\mathcal{CLN}$ then
$$
\mathrm{Def}(\mathbf{E}_1,\mathbf{E}_2) (R) \iso \Hom_{\mathcal{CLN}}(W,R)
$$
consists of a single point.  From this it follows easily that the pair $(\mathbf{E}_1,\mathbf{E}_2)$
admits a unique lift to characteristic $0$, and that the reduction map
$$
[\mathcal{X}(\C_p)] \map{} [ \mathcal{X}(\F_p^\alg) ]
$$
is a bijection.  As explained in Remark \ref{four orbits} there are $4\cdot |\Gamma|$ isomorphism
classes of CM pairs over $\C_p$, and hence also $4\cdot |\Gamma|$ isomorphism classes of CM pairs
over $\F_p^\alg$.  If $(\mathbf{E}_1,\mathbf{E}_2)$ has reflex prime $\mathfrak{p}$ then we may consider the pair
$(\mathbf{E}_1,\mathbf{E}_2')$ obtained by replacing $\mathbf{E}_2=(E_2,\kappa_2)$ by the pair
$\mathbf{E}'_2=(E_2, \kappa'_2)$, where $\kappa'_2(x) = \kappa_2(\overline{x})$.  It is easy to check that
the CM pair $(\mathbf{E}_1,\mathbf{E}_2')$ has reflex prime $\mathfrak{p}'\not=\mathfrak{p}$, and hence
 half of the  $4\cdot |\Gamma|$ isomorphism classes of supersingular CM pairs have
reflex prime $\mathfrak{p}$, and half have reflex prime $\mathfrak{p}'$.

Now suppose that $p$ is ramified in one of $K_1$ or $K_2$, and for simplicity suppose it is $K_1$.
Then $\mathrm{Def}(\mathbf{E}_1,\mathbf{E}_2)$ represented by
$\mathcal{W}_{L_1}\widehat{\otimes}_W \mathcal{W}_{L_2} \iso \mathcal{W}_{L_1}$, and so for any object
$R$ of $\mathcal{CLN}$
$$
\mathrm{Def}(\mathbf{E}_1,\mathbf{E}_2) (R) \iso \Hom_{\mathcal{CLN}}(\mathcal{W}_{L_1},R).
$$
Assuming that $R$ is large enough to contain a subring isomorphic to $\mathcal{W}_{L_1}$,
the set on the right  consists of two points
(which are interchanged by pre-composing with the nontrivial element of
$\Aut(\mathcal{W}_{L_1}/W)$).   It follows from this that $(\mathbf{E}_1,\mathbf{E}_2)$ admits
precisely two nonisomorphic deformations to characteristic $0$, and that the reduction map
$$
[\mathcal{X}(\C_p)] \map{} [ \mathcal{X}(\F_p^\alg) ]
$$
is two-to-one.  As in the preceding paragraph, there are $4\cdot |\Gamma|$ isomorphism classes of
CM pairs over $\C_p$, and therefore there are half as many over $\F_p^\alg$.  As $p$ is ramified in
$F$ ($p\co_F=\mathfrak{p}^2$), we find that there are $2\cdot |\Gamma|$ isomorphism classes of CM
pairs over $\F_p^\alg$,  all having  reflex prime $\mathfrak{p}$.
\end{proof}

Fix a CM pair $(\mathbf{E}_1,\mathbf{E}_2)$ over $\F_p^\alg$ and let
 $\mathfrak{p}$ be the reflex prime of $(\mathbf{E}_1,\mathbf{E}_2)$.  Given any nonzero
 $j\in L(\mathbf{E}_1,\mathbf{E}_2)$ let  $\mathrm{Def}(\mathbf{E}_1,\mathbf{E}_2,j)$
be the functor that assigns to every object $R$ of $\mathcal{CLN}$ the set of isomorphism classes of deformations
of $ (\mathbf{E}_1,\mathbf{E}_2, j)$ to $R$.
  As above choose  an  isomorphism of $p$-divisible groups $E_i[p^\infty] \iso \mathfrak{g}$.  Such
  choices determine an isomorphism
  $$
  L_p(\mathbf{E}_1,\mathbf{E}_2) \iso \Delta
  $$
  of $\Z_p$-modules, and allow us to view $\deg$ as a $\Z_p$-quadratic form  on $\Delta$.  The quadratic form
  $\deg$ on $\Delta$   is  a $\Z_p^\times$-multiple of the   reduced norm.

  \begin{Lem}\label{Lem:local ring unr}
If  $p$ is inert in both $K_1$ and $K_2$ then the deformation functor $\mathrm{Def}(\mathbf{E}_1,\mathbf{E}_2,j)$
is represented by a local Artinian $W$-algebra of length
\begin{equation}\label{unr length}
  \frac{\ord_\mathfrak{p}(\cmdeg(j)) +1   }{2}.
\end{equation}
  \end{Lem}

  \begin{proof}
    The isomorphisms $E_i[p^\infty] \iso \mathfrak{g}$
  may be chosen so that $\kappa_1:\co_{K_1,p} \map{}\Delta$ and $\kappa_2:\co_{K_2,p} \map{}\Delta$ have the
  same image $\co_L\iso \Z_{p^2}$, and we may fix a uniformizing parameter $\Pi\in \Delta$ with the property
  $u \Pi = \Pi u^\iota$ for every $u\in \co_L$.  There is then a decomposition of left $\co_L$-modules
  $$
  \Delta = \Delta_+ \oplus \Delta_-
  $$
  where $\Delta_+=\co_L$ and $\Delta_-=\co_L\Pi$, and this decomposition
 is orthogonal with respect to the quadratic form $\deg$ on $\Delta$.    Now define
$f_\pm:\co_{K,p}\map{} \co_L$  by
\begin{align*}
f_+(x_1\otimes x_2) &=  \kappa_2(x_2) \kappa_1(\overline{x}_1)\\
f_-(x_1\otimes x_2) &=    \kappa_2(x_2) \kappa_1(x_1)
\end{align*}
and denote by $\Psi$ the isomorphism
$$
\Psi=f_+\times f_- :\co_{K,p}  \map{}\co_L\times\co_L.
$$
Let $\mathfrak{p}_-=\mathfrak{p}$ be the reflex prime of the pair $(\mathbf{E}_1,\mathbf{E}_2)$, and let
$\mathfrak{p}_+$ be the other prime of $F$ above $p$.  The map $f_\pm$ factors through the completion
of $\co_{K,p}$ at the prime $\mathfrak{p}_\pm$.
The action (\ref{K action})  of $\co_K$ on $\Delta$ is by (for $j=j_++j_-$)
$$
x \action j =  f_+(x) j_+ +  f_-(x) j_-
$$
and the $\co_{F,p}$-quadratic form $\cmdeg$ on $\Delta$ takes the explicit form
$$
\Psi(\cmdeg(j)) = ( \deg(j_+), \deg(j_-)).
$$
It follows that
\begin{align*}
\ord_{\mathfrak{p}_+} (\cmdeg(j)) &= \ord_p(\deg(j_+)) \\
\ord_{\mathfrak{p}_-} (\cmdeg(j)) &= \ord_p(\deg(j_-)) .
\end{align*}
In particular for any positive $k\in \Z$ we have
\begin{align*}
j \in \co_L + p^{k-1} \Delta & \iff j_-\in p^{k-1} \co_L\Pi \\
&\iff \ord_p(\deg(j_-))  \ge 2k-1 \\
&\iff \frac{\ord_\mathfrak{p}(\cmdeg(j)) +1   }{2} \ge k.
\end{align*}

The deformation functor $\mathrm{Def}(\mathfrak{g},\co_L)$ is represented by $W$, and hence
$$
\mathrm{Def}(\mathbf{E}_1,\mathbf{E}_2)\iso \mathrm{Def}(\mathfrak{g},\co_L) \times \mathrm{Def}(\mathfrak{g},\co_L)
$$
is represented by $W\widehat{\otimes}_W W\iso W$.
If we let $\mathfrak{G}$ be the universal deformation of $\mathfrak{g}$ with its $\co_L$-action
then the  $p$-divisible group of the universal deformation of  $(\mathbf{E}_1,\mathbf{E}_2)$ is identified with
$(\mathfrak{G},\mathfrak{G})$, and  the deformation functor of the triple $(\mathbf{E}_1,\mathbf{E}_2,j)$
is represented by $W/p^kW$ where $k$ is the largest integer for which $j$ lifts to an
endomorphism of the reduction  $\mathfrak{G} \otimes_W W/p^kW$.  Combining (\ref{gross deform}) with the calculation of preceding paragraph, this $k$
is given by the formula (\ref{unr length}).
\end{proof}

\begin{Lem}\label{Lem:local ring ram}
If  $p$ is ramified in one of $K_1$ or $K_2$ then the deformation functor
$\mathrm{Def}(\mathbf{E}_1,\mathbf{E}_2,j)$ is represented by a local Artinian $W$-algebra of
length
$$
  \frac{\ord_\mathfrak{p}(\cmdeg(j)) + \ord_{\mathfrak{p}}(\mathfrak{D})+ 1   }{2}.
$$
\end{Lem}

\begin{proof}
The prime $p$ is ramified in one of $K_1$ or $K_2$ and is inert in the other.
For simplicity let us assume that $p$ is ramified in $K_2$ and inert in $K_1$.  Let $\co_{L_i}$ be the image
of $\kappa_i:\co_{K_i,p}\map{}\Delta$.  If we choose a uniformizing parameter $\varpi\in \co_{K_2,p}$ then
$\pi=1\otimes\varpi$ is a uniformizing parameter of $\co_{K,p}$ and $\Pi = \kappa_2(\varpi)$ is a uniformizing
parameter of both $\co_{L_2}$ and $\Delta$.   The action (\ref{K action}) of $\pi \in \co_{K,p}$ on
$L_p(\mathbf{E}_1,\mathbf{E}_2) \iso \Delta$ is simply left multiplication by $\Pi$.  Let $\beta\in \co_{F,p}$
be as in Lemma \ref{Lem:CM model I}, and choose an $\co_{K,p}$-linear  isomorphism
$$
\big( \co_{K,p}, \beta\cdot \mathrm{Nm}_{K/F} \big) \iso \big( \Delta, \cmdeg \big).
$$
As this isomorphism carries $\pi^k \co_{K,p}$ isomorphically to $\Pi^k\Delta$, if we view
$j$ as both an element of $\Delta$ and an element of $\co_{K,p}$ we have
\begin{align*}
\ord_\Delta(j) & =  \ord_\mathfrak{p}(j)  \\
& = \frac{ \ord_\mathfrak{p}(\mathrm{Nm}_{K/F}(j) ) }{2} \\
& = \frac{ \ord_\mathfrak{p}(\beta\cdot  \mathrm{Nm}_{K/F}(j) ) +\ord_\mathfrak{p}(\mathfrak{D}) -1 }{2} \\
& = \frac{ \ord_\mathfrak{p}(\cmdeg(j) ) +\ord_\mathfrak{p}(\mathfrak{D}) -1 }{2} .
\end{align*}

 The functor $\mathrm{Def}(\mathfrak{g},\co_{L_i})$ is represented by $\mathcal{W}_{L_i}$.
 In particular  $\mathrm{Def}(\mathfrak{g},\co_{L_1})$ is represented by $W$, and so
 $\mathfrak{g}$, with its $\co_{L_1}$-action, admits a unique deformation to $\mathcal{W}_{L_2}$
 corresponding to the unique element of  $\Hom_\mathcal{CLN}(W,\mathcal{W}_{L_2})$.
 Call this deformation $\mathfrak{G}^{(1)}$,   let
$\mathfrak{G}^{(2)}$ be the universal deformation of $\mathfrak{g}$ to $\mathcal{W}_{L_2}$, and let
$\mathfrak{G}^{(i)}_m$ be the reduction of $\mathfrak{G}^{(i)}$ to $\mathcal{W}_{L_2}/(\pi^m)$.    The deformation
functor $\mathrm{Def}(\mathbf{E}_1,\mathbf{E}_2)$ is then represented by
$W\widehat{\otimes}_W\mathcal{W}_{L_2} \iso \mathcal{W}_{L_2}$,  the $p$-divisible group of the universal deformation
of the pair $(\mathbf{E}_1,\mathbf{E}_2)$ is $(\mathfrak{G}^{(1)},\mathfrak{G}^{(2)})$, and
the functor $\mathrm{Def}(\mathbf{E}_1,\mathbf{E}_2,j)$ is represented by $\mathcal{W}_{L_2}/(\pi^k)$
where $k$ is the largest integer such that the homomorphism $j:\mathfrak{g}\map{}\mathfrak{g}$ lifts to a homomorphism
$\mathfrak{G}^{(1)}_k \map{}\mathfrak{G}^{(2)}_k$.

Factor $j=\Pi^m \cdot u$ with $u\in \Delta^\times$ and $m=\ord_\Delta(j)$.
 Suppose $u$ lifts to a homomorphism $\mathfrak{G}^{(1)}_k\map{}\mathfrak{G}^{(2)}_k$.
 This lift is necessarily an isomorphism, and as $\Pi\in \co_{L_2}$ lifts to an endomorphism of
 $\mathfrak{G}^{(2)}$ we find that  $u^{-1}\circ \Pi \circ u$ also  lifts to an
an endomorphism of $\mathfrak{G}_k^{(1)}$.  But $\co_{L_1}$ and  $u^{-1}\circ \Pi \circ u$ generate
 all of $\Delta$ as a $\Z_p$-algebra (as do any subalgebra isomorphic to $\Z_{p^2}$ and
 uniformizer of $\Delta$).  Thus  every element of $\Delta$ lifts to an endomorphism of $\mathfrak{G}_k^{(1)}$.
It follows from (\ref{gross deform}) that the functor $\mathrm{Def}(\mathfrak{g},\Delta)$ is
represented by $\F_p^\alg$, and we deduce that $k=1$.  Thus $u$ lifts to an
endomorphism $\mathfrak{G}^{(1)}_1\map{}\mathfrak{G}^{(2)}_1$
but not to $\mathfrak{G}^{(1)}_2\map{}\mathfrak{G}^{(2)}_2$.   It now follows immediately from
\cite[Proposition 5.2]{ARGOS-8} that $j=\Pi^m\cdot u$ lifts to an endomorphism
$\mathfrak{G}_{m+1}^{(1)} \map{} \mathfrak{G}_{m+1}^{(2)}$ but not to
$\mathfrak{G}_{m+2}^{(1)} \map{} \mathfrak{G}_{m+2}^{(2)}$, and so
the functor $\mathrm{Def}(\mathbf{E}_1,\mathbf{E}_2,j)$ is represented by $\mathcal{W}_{L_2}/(\pi^{m+1})$ where
$$
m+1=\ord_\Delta(j)+1 =  \frac{ \ord_\mathfrak{p}(\cmdeg(j) ) +\ord_\mathfrak{p}(\mathfrak{D}) +1 }{2}.
$$
\end{proof}

Let $\mathcal{C}$ be any algebraic stack over $\Spec(\Z)$ and suppose
 $x\in \mathcal{C}(\F_p^\alg)$ is a geometric point.  An \emph{\'etale neighborhood} of
$x$ is a commutative diagram in the category of algebraic stacks
\begin{equation}\label{etale}
\xymatrix{
 & { U } \ar[d]  \\
{\Spec(\F_p^\alg) } \ar[r]_x  \ar[ur]^{\tilde{x}} & {\mathcal{C}}
}
\end{equation}
in which $U$ is a scheme and the vertical arrow is an \'etale morphism.  The \emph{strictly
Henselian local ring} of $\mathcal{C}$ at $x$ is the direct limit
$$
\co^\mathrm{sh}_{\mathcal{C},x} = \varinjlim_{(U,\tilde{x})} \co_{U,\tilde{x}}
$$
over all \'etale neighborhoods of $x$, where $\co_{U,\tilde{x}}$ is the usual local ring of the scheme
$U$ at (the image of) $\tilde{x}$.  As the name suggests, $\co^\mathrm{sh}_{\mathcal{C},x}$
is a strictly Henselian local ring, and has residue field $\F_p^\alg$.   If $T\subset W$ is a subring
that is \'etale as a $\Z$-algebra then $U=\mathcal{C}\times_{\Spec(\Z)} \Spec(T)$ is naturally an \'etale neighborhood
of $x$, and so $\co^\mathrm{sh}_{\mathcal{C},x}$ is a $T$-algebra.  As the union of all such $T$
is  dense in $W$, the completed strictly Henselian local ring
$\widehat{\co}^\mathrm{sh}_{\mathcal{C},x}$ is naturally a $W$-algebra.

\begin{Prop}\label{Prop:local length}
Suppose $\alpha\in F^\times$ satisfies $\Sppt(\alpha)=\{p\}$.  For any
$x \in \mathcal{X}_\alpha(\F_p^\alg)$   the strictly Henselian local ring of
$\mathcal{X}_\alpha$   at $x$ is Artinian of length
$$
\nu_{\mathfrak{p}}(\alpha)= \frac{\ord_{\mathfrak{p}}(\alpha\mathfrak{D})+1}{2}
$$
where $\mathfrak{p}$ is the reflex prime of the pair  $(\mathbf{E}_1,\mathbf{E}_2)$ underlying the
triple $x=(\mathbf{E}_1,\mathbf{E}_2, j)$.
\end{Prop}

\begin{proof}
  Suppose $R$ is any local Artinian $W$-algebra with residue field $\F_p^\alg$, and
$z \in \mathrm{Def}(\mathbf{E}_1,\mathbf{E}_2,j)(R)$.  Then $z$ determines  a point
$z\in \mathcal{X}_\alpha(R)$ whose image under the reduction map
$\mathcal{X}_\alpha(R)\map{}\mathcal{X}_\alpha(\F_p)$ is $x$, and so we have
a commutative diagram
$$
\xymatrix{
{ \Spec(R) } \ar[rd]^z  \\
{\Spec(\F_p^\alg) } \ar[r]_x  \ar[u] & {\mathcal{X}_\alpha}.
}
$$
Given an  \'etale neighborhood (\ref{etale}) of $x$ there is a unique $\tilde{z}:\Spec(R)\map{} U$
making the diagram
$$
\xymatrix{
{ \Spec(R) } \ar[rd] \ar[r]^{\tilde{z}} & {U} \ar[d]\\
{\Spec(\F_p^\alg) } \ar[r]_x\ar[ur]  \ar[u] & {\mathcal{X}_\alpha}
}
$$
commute (the morphism $U\map{}\mathcal{X}_\alpha$ is formally \'etale).
The morphism of schemes  $\tilde{z}$ induces a ring homomorphism  $\tilde{z}:\co_{U,\tilde{x}}\map{} R$, and
by varying the \'etale  neighborhood we obtain a map
$\tilde{z}: \co^\mathrm{sh}_{\mathcal{X}_\alpha,x} \map{}R$ that  induces the identity  on  residue fields.
 In particular $\tilde{z}$ extends uniquely to
$\tilde{z} \in \Hom_{\mathcal{CLN}}(  \widehat{\co}^\mathrm{sh}_{\mathcal{X}_\alpha,x} , R)$.
It is now easy to check that the construction $z\mapsto \tilde{z}$ establishes a  bijection
$$
\mathrm{Def}(\mathbf{E}_1,\mathbf{E}_2,j)(R) \map{}
\Hom_{\mathcal{CLN}}(  \widehat{\co}^\mathrm{sh}_{\mathcal{X}_\alpha,x} , R)
$$
for every Artinian $R$ in $\mathcal{CLN}$, and that  the functor
$\mathrm{Def}(\mathbf{E}_1,\mathbf{E}_2,j)$ is represented by the completed strictly Henselian
local ring $ \widehat{\co}^\mathrm{sh}_{\mathcal{X}_\alpha,x}$.  The claim is now immediate
from Lemmas \ref{Lem:local ring unr} and \ref{Lem:local ring ram}.
\end{proof}


\subsection{Final formula}
\label{ss:arithmetic section}


Recall our notation:  $\chi$ is the quadratic Hecke character associated to the extension $K/F$;
if $\alpha\in F^\times$ is totally positive then
 $\Diff(\alpha)$ is  the set of all finite primes $\mathfrak{p}$ of $F$ such that
$\chi_\mathfrak{p}(\alpha\mathfrak{D} )=-1$;
for a fractional $\co_F$-ideal $\mathfrak{b}$ we denote by $\rho(\mathfrak{b})$  the number of ideals
$\mathfrak{B}\subset \co_E$ satisfying $\mathrm{Nm}_{K/F}(\mathfrak{B}) = \mathfrak{b}$.

\begin{Thm}\label{Thm:local length}
Suppose $\alpha\in F^\times$.  If $|\Diff(\alpha)| > 1$ then  $\mathcal{X}_\alpha=\emptyset$.
If $\Diff(\alpha)=\{\mathfrak{p}\}$ and $p\Z=\mathfrak{p}\cap \Z$ then
$\mathcal{X}_\alpha$ is supported in characteristic $p$,
the strictly Henselian local ring of every geometric point $x\in  \mathcal{X}_\alpha(\F_p^\alg)$ is
Artinian of length
$$
\nu_{\mathfrak{p}}(\alpha)= \frac{1}{2}\cdot \ord_{\mathfrak{p}}(\alpha\mathfrak{p}\mathfrak{D}),
$$
and the CM pair $(\mathbf{E}_1,\mathbf{E}_2)$ underlying the triple $x=(\mathbf{E}_1,\mathbf{E}_2,j)$
has reflex prime $\mathfrak{p}$.
\end{Thm}

\begin{proof}
Suppose $\mathcal{X}_\alpha\not=\emptyset$, and recall from
Corollary \ref{Cor 2.2.4} that $\mathcal{X}_\alpha$ is supported in a single nonzero characteristic $p$.
Fix a triple  $(\mathbf{E}_1,\mathbf{E}_2,j) \in \mathcal{X}_\alpha(\F_p^\alg)$
and let $\mathfrak{p}$ be the reflex prime of $(\mathbf{E}_1,\mathbf{E}_2)$.
As $p$ is nonsplit in both $K_1$ and $K_2$ it follows that
$\mathfrak{p}$ is inert in $K$, and so for every finite place $\mathfrak l$ of $F$
$$
\chi_\mathfrak l(\mathfrak{p})=\begin{cases}
-1 & \hbox{if }\mathfrak l=\mathfrak{p} \\
1 &\hbox{otherwise}.
\end{cases}
$$
On the other hand Theorem \ref{Thm:quadratic model} implies that $(\widehat{K},\beta x\overline{x})$
represents $\alpha$ for any $\beta\in \widehat{F}^\times$ satisfying
$\beta\widehat{\co}_F = \mathfrak{p}\mathfrak{D}^{-1}\widehat{\co}_F$.  This implies that
$\chi_\mathfrak l(\alpha) = \chi_\mathfrak l(\mathfrak{p}\mathfrak{D}^{-1})$ for every finite place $\mathfrak l$ of $F$, and we deduce that
$\Diff(\alpha)=\{\mathfrak{p}\}$.   We have now shown that if $\mathcal{X}_\alpha\not=\emptyset$
and is supported in characteristic $p$ then $\Diff(\alpha)$ contains a single prime $\mathfrak{p}$.  This
prime $\mathfrak{p}$ lies above $p$ and is equal to the reflex prime of every triple
$(\mathbf{E}_1,\mathbf{E}_2,j) \in \mathcal{X}_\alpha(\F_p^\alg)$.
The stated formula for the length of the strictly Henselian local ring of a geometric
point is now just a restatement of Proposition \ref{Prop:local length}.
\end{proof}

\begin{Thm}\label{Thm:final formula}
Suppose $\alpha\in F^\times$ is totally positive.
If $\Diff(\alpha)=\{\mathfrak{p}\}$ then
$$
\deg(\mathcal{X}_\alpha) = \frac{1}{2}\cdot \log(p) \cdot
\ord_\mathfrak{p}(\alpha\mathfrak{p}\mathfrak{D}) \cdot
\rho(\alpha\mathfrak{D}\mathfrak{p}^{-1})
$$
where $p\Z=\mathfrak{p}\cap\Z$, and the left hand side is the Arakelov degree of (\ref{arakelov degree}).
\end{Thm}

\begin{proof}
Let $[\mathcal{X}(\F_p^\alg)]_\mathfrak{p}$
denote the subset of $[\mathcal{X}(\F_p^\alg)]$ consisting of isomorphism classes of
supersingular CM pairs with
reflex prime $\mathfrak{p}$.  Combining  Theorem \ref{Thm:local length} with
Lemma \ref{C stab}   results in
\begin{align*}
\deg(\mathcal{X}_\alpha)  & = \log(p) \sum_{x\in [\mathcal{X}_\alpha(\F_p^\alg)] }
e_x^{-1} \cdot\length( \co^\mathrm{sh}_{  \mathcal{X}_\alpha, x  } ) \\
& = \log(p) \cdot \nu_\mathfrak{p}(\alpha)
\sum_{( \mathbf{E}_1,\mathbf{E}_2,j)\in [\mathcal{X}_\alpha(\F_p^\alg)]}
\frac{1}{ | \Aut(\mathbf{E}_1,\mathbf{E}_2,j) |} \\
&=
\log(p) \cdot \nu_\mathfrak{p}(\alpha)
  \sum_{ (\mathbf{E}_1,\mathbf{E}_2) \in [\mathcal{X}(\F_p^\alg)]_\mathfrak{p} }
  \sum_{ \substack{ j \in  L(\mathbf{E}_1,\mathbf{E}_2) \\  \cmdeg(j) = \alpha} }
  \frac{1}{ \mathbf{w}_1\mathbf{w}_2}.
  \end{align*}
 Now applying Proposition \ref{Prop:unfolding} and Theorem \ref{Thm:global orbital} to the
 final expression we find
 \begin{align*}
  \deg(\mathcal{X}_\alpha)
  &=
 \log(p) \cdot \frac{ \nu_\mathfrak{p}(\alpha) }{2\cdot |\Gamma| }
  \sum_{ (\mathbf{E}_1,\mathbf{E}_2) \in [\mathcal{X}(\F_p^\alg)]_\mathfrak{p}  }
 \prod_{\ell < \infty} O_\ell(\alpha,\mathbf{E}_1,\mathbf{E}_2) \\
 &=
 \log(p) \cdot \frac{ \nu_\mathfrak{p}(\alpha) }{2\cdot |\Gamma| }
  \sum_{ (\mathbf{E}_1,\mathbf{E}_2) \in [\mathcal{X}(\F_p^\alg)]_\mathfrak{p}  }
\rho(\alpha\mathfrak{D}\mathfrak{p}^{-1}),
\end{align*}
and applying Proposition \ref{Prop:naked count} to this equality results in
\begin{align*}
 \deg(\mathcal{X}_\alpha)
&=
 \log(p) \cdot  \nu_\mathfrak{p}(\alpha) \cdot \rho(\alpha\mathfrak{D}\mathfrak{p}^{-1})
\end{align*}
as desired.
\end{proof}

If $\mathfrak{b}$ is any fractional $\co_{F}$-ideal and $p$ is a prime that is nonsplit in both $K_1$
and $K_2$ we set
$$
f_p(\mathfrak{b}) = \sum_{\mathfrak{p}} \ord_{\mathfrak{p}}(\mathfrak{b}\mathfrak{p}) \cdot \rho(\mathfrak{b}\mathfrak{p}^{-1})
$$
where the sum is over the primes $\mathfrak{p}$ of $F$ above $p$.  If $p$ is a prime that splits in either
$K_1$ or $K_2$ we set $f_p(\mathfrak{b})=0$. It is clear from the definition that $f_p(\mathfrak{b})=0$
unless $\mathfrak{b}\subset \co_F$.

\begin{Thm} \label{Thm:singular moduli}
For any totally positive $\alpha\in F^\times$ the stack $\mathcal{X}_\alpha$ has Arakelov degree
\begin{equation}\label{sing mod}
\deg(\mathcal{X}_\alpha)  =  \frac{1}{2} \cdot \sum_p f_p(\alpha\mathfrak{D}) \log(p).
\end{equation}
\end{Thm}

\begin{proof}
Suppose  $\mathfrak{q}\in \Diff(\alpha)$.   Then $\mathfrak{q}$ is inert in $K$ and
$\ord_\mathfrak{q}(\alpha\mathfrak{D})$ is odd.  If $\mathfrak{p}$ is any prime of $F$ that
 satisfies $\rho(\alpha\mathfrak{D}\mathfrak{p}^{-1})\not=0$, then
$\ord_{\mathfrak{q}}(\alpha\mathfrak{D}\mathfrak{p}^{-1})$ is even and so $\mathfrak{p}=\mathfrak{q}$.
Thus
$$
\frac{1}{2} \cdot \sum_p f_p(\alpha\mathfrak{D}) \log(p) =
\frac{1}{2} \cdot  \ord_\mathfrak{q}(\alpha\mathfrak{D}) \rho(\alpha\mathfrak{D}\mathfrak{q}^{-1}) \log(q)
$$
where $q\Z=\mathfrak{q}\cap \Z$.   In particular, if $\Diff(\alpha)=\{\mathfrak{q}\}$ then the
desired equality follows from Theorem \ref{Thm:final formula}.

Now suppose  $|\Diff(\alpha)| >1 $. The claim is that both sides of (\ref{sing mod}) are  equal to $0$.
The vanishing of the left hand side is precisely the final claim of Theorem \ref{Thm:local length}.
To see that the right hand side vanishes,  suppose  $\mathfrak{r},\mathfrak{q} \in \Diff(\alpha)$ with
$\mathfrak{r}\not=\mathfrak{q}$.   The hypothesis  $\mathfrak{r}\in \Diff(\alpha)$
implies that $\ord_\mathfrak{r}(\alpha\mathfrak{D})$ is odd.  Therefore
$\ord_\mathfrak{r}(\alpha\mathfrak{D} \mathfrak{q}^{-1} )$ is also odd  and so
$\rho(\alpha\mathfrak{D}\mathfrak{q}^{-1})=0$.   The calculation of the previous paragraph
now shows that the right hand side of (\ref{sing mod}) is equal to $0$.
\end{proof}

\begin{Cor}[Gross-Zagier]\label{Cor:GZ}
 For every $m\in \Q^\times$
$$
\deg(\mathcal{T}_m)  =  \frac{1}{2}    \sum_{
\substack{\alpha\in \mathfrak{D}^{-1} \\
\mathrm{Tr}_{F/\Q}(\alpha) = m \\ \alpha \gg 0 } }  \sum_p
f_p(\alpha\mathfrak{D})  \log(p) .
$$
\end{Cor}

\begin{proof}
This is immediate from Theorem \ref{Thm:singular moduli} and the decomposition (\ref{moduli decomp}).
\end{proof}


\section{Eisenstein Series}\label{sect3}


 Let $\psi_{\mathbb Q} =\prod_{p} \psi_{\mathbb Q_p}$ be the unique
 unramified additive character $\mathbb Q \backslash \A\map{}\C^\times$ satisfying $\psi_{\mathbb R}(x)
 =e(x) =e^{2 \pi i x}$, and let $\psi_F = \psi_{\mathbb Q} \circ
 \tr_{F/\mathbb Q}$.  Denote by  $\sigma_1$ and $\sigma_2$  the two real
 embeddings of $F$,  let $\chi$ be the quadratic Hecke character of
 $F$ associated to $K$, and  let $\sqrt D \in F$ be a fixed square root of $D$, so that $\sqrt D\co_F =\mathfrak D$.

\subsection{Incoherent quadratic spaces and Hecke's Eisenstein series}

  Let $W=K$ viewed as an $F$-quadratic space  with the quadratic form
  $Q(x) =\frac{1}{\sqrt D} x \bar x$.   Define a collection $\{\mathcal{C}_v\}$ of $F_v$-quadratic spaces, one for
  each place $v$ of $F$, by taking $\mathcal C_{\mathfrak p} =W_{\mathfrak p}$ for every
  finite prime $\mathfrak p$ of $F$, and taking $\mathcal C_{\sigma_i}$ to be of
  signature $(2, 0)$ for the two infinite primes $\sigma_1$ and  $\sigma_2$ of
  $F$.   The product $\mathcal{C}=\prod_v\mathcal{C}_v$ is then an \emph{incoherent} quadratic space
  over $\A_{F}$ in the sense of \cite[Definition 2.1]{KuAnnals}.  For each place $v$ of $F$ let $\Sch(\mathcal{C}_v)$ be the space of
  Schwartz functions on $\mathcal{C}_v$, and let $\Sch(\mathcal{C}) = \otimes_v \Sch(\mathcal{C}_v)$ be the space of
  Schwartz functions on $\mathcal{C}$.  By means of the Weil representation $\omega=\omega_{\mathcal C,
  \psi_F}$ (see for example \cite{KuSplit}), one has an $\SL_2(\A_F)$-equivariant map
  $$
  \lambda:   \Sch(\mathcal C)  \rightarrow
  I(0, \chi), \quad \lambda(\phi)(g) =\omega(g)\phi(0).
  $$
  Here $I(s, \chi)=\hbox{Ind}_{B(\A_F)}^{\SL_2(\A_F)} (\chi)$ is the induced
  representation of $\SL_2(\A_F)$ consisting of smooth functions
  $\Phi(g, s)$ on $\SL_2(\A_F)$ satisfying
$$
\Phi(n(b) m(a) g, s) =\chi(a) |a|_{\A}^{s+1} \Phi(g, s),  \quad b
\in \A_F, a\in \A_F^\times,
$$
and
$$
 B=N M=\{ n(b) m(a) =\kzxz {1} {b} {0} {1} \kzxz  {a} {0} {0}
 {a^{-1}}:\, b \in F, a \in F^\times\}
 $$
viewed as an algebraic group over $F$.  We say $\Phi \in I(s, \chi)$ is \emph{standard} if $\Phi(g, s)$ is
independent of $s$ for $g$ in the maximal compact subgroup
$\SL_2(\widehat{\co}_F) \times \SO_2(F_\infty)$. We say $\Phi$ is
factorizable if $\Phi =\otimes  \Phi_v$ is the product of local
sections.  For a factorizable standard section $\Phi \in I(s,\chi)$ the Eisenstein series
 $$
 E(g, s, \Phi) = \sum_{\gamma \in B(F) \backslash \SL_2(F)}
 \Phi(\gamma g, s)
 $$
 is absolutely convergent when $\mathrm{Re } ( s) $ is sufficiently large,  and has meromorphic
 continuation with a functional equation in $s \mapsto  -s$.
 Moreover, it is holomorphic along the unitary axis $\mathrm{Re }( s) =0$.

\begin{Rem} Following \cite[Definition 5.1, (5.4)]{KuAnnals}, we denote by $\Diff(\mathcal C, \alpha)$ the
set of places $v$ of $F$ at which  $\mathcal C_v$ does not represent $\alpha$. Then
$\Diff(\mathcal C, \alpha) =\Diff(\alpha)$ for every totally positive $\alpha \in F$, where $\Diff(\alpha)$
is the set defined in the introduction.
\end{Rem}

\begin{Rem}
The incoherent  quadratic space $\mathcal C$ is closely related to the quadratic spaces studied in Section  \ref{sect2}.
Let $p$ be a prime that is nonsplit in $K_1$ and $K_2$, let $(\mathbf E_1, \mathbf E_2)$ be a  supersingular CM pair  over
$\mathbb F_p^\alg$, and let $\mathfrak p$ be the reflex prime of $(\mathbf E_1, \mathbf E_2)$
in the sense of Definition \ref{Def:reflex prime}.  By Theorem \ref{Thm:quadratic model},  for every place $v$ of $F$
$$
\big( V(\mathbf E_1,\mathbf E_2)\otimes_F F_v , \cmdeg\big)  \iso \mathcal C_v  \iff v\not=\mathfrak{p}
$$
 (see also Lemma \ref{lem3.1.3} below).
\end{Rem}

 For $\phi \in \Sch(\mathcal C)$ let $\Phi_\phi \in I(s, \chi)$ be the
 standard section  associated to $\lambda(\phi)$, characterized by
 $\Phi_\phi(g, 0) =\lambda(\phi)$, and abbreviate
 $$
 E(g, s, \phi) =E(g, s, \Phi_\phi).
 $$

 We now choose a particular $\phi_v^+\in \Sch(\mathcal{C}_v)$ for every place $v$ of $F$:
 for a finite prime $\mathfrak p$  set   $\phi_{\mathfrak p}^+ =\mathbf{1}_{\co_{K, \mathfrak p}}$, and for $l\in \{1,2\}$ set
$\phi_{\sigma_l}^+ = e^{- 2 \pi Q_{\sigma_l}(x)}$, where $Q_{\sigma_l}$ is the quadratic form on 
$\mathcal C_{\sigma_l}$. We  define  $\phi^{\mathcal C}  =\otimes_{v} \phi_v^+ \in \Sch(\mathcal C)$ and let 
$$\Phi^\mathcal C =\otimes_v \Phi_v^+ \in I(s, \chi)$$ be
the associated standard section, where  $\Phi_v^+$ is  the standard section associated to $\phi_v^+$. The
following is well-known.

\begin{Lem}  \label{lem3.0.1}\

\begin{enumerate}
\item
For all $s$ and $l\in \{1,2\}$,    $\Phi_{\sigma_l}^+(g,s)$ is the
normalized eigenfunction of $\SO_2(F_{\sigma_l})$ of
weight $1$,  i.e. $\Phi_{\sigma_i}^+(1,s)=1$ and
$$
\Phi_{\sigma_i}^+(g k_\theta,s) = \Phi_{\sigma_i}^+(g,s) \cdot e^{i \theta},
\quad k_\theta=\kzxz {\cos \theta} {\sin \theta} {-\sin \theta}
{\cos \theta} .
$$

\item
 For all $s$ and all finite primes $\mathfrak p$, $\Phi_\mathfrak p^+(g,s)$ is the spherical section in $I(s,
\chi_\mathfrak p)$, i.e. $\Phi_\mathfrak p^+(1,s) =1$ and
$$
\Phi_{\mathfrak p}^+(g k,s) =\Phi_{\mathfrak p}^+(g,s),  \quad k \in \SL_2(\co_{F_\mathfrak p}).
$$
\end{enumerate}
\end{Lem}
\begin{proof}  Claim (1) follows from \cite[Lemma 1.2]{Shi}. For (2), let 
$\psi_{\mathfrak p}'(x) = \psi_{\mathfrak p}(\frac{1}{\sqrt D} x)$
so that  $\psi_{\mathfrak p}'$ is unramified. Let $V_\mathfrak p'=K_\mathfrak p$ with $Q'(x) =x \bar x$ for 
$x \in K_\mathfrak p$. Then
$$
\omega_{V_\mathfrak p, \psi_\mathfrak p} = \omega_{V_\mathfrak p', \psi_\mathfrak p'}.
$$
Now (2) follows from \cite[Proposition 2.1]{YaValue}.
\end{proof}

Let $\mathbb{H}$ be the complex upper half-plane. For $\tau=(\tau_1, \tau_2) \in \mathbb H^2$ write
$\tau_l=u_l +i v_l$,  set
$$
g_{\tau_l}=n(u_l)\cdot  m(\sqrt{v_l})\in \SL_2(\R),
$$
and view   $g_\tau =(g_{\tau_1}, g_{\tau_2})$ as an element of
$$
\SL_2(F_{\sigma_1}) \times \SL_2(F_{\sigma_2})\subset\SL_2(\A_F).
$$
 Let
\begin{equation} \label{eq: Eisenstein}
E^*(\tau, s, \phi^{\mathcal C}) =\Lambda(s+1, \chi) E(\tau, s, \phi^{\mathcal C}) =\Lambda(s+1, \chi)(v_1
v_2)^{-\frac{1}2} E(g_\tau, s, \Phi^{\mathcal C})
\end{equation}
be the normalized Eisenstein series of weight $1$, where
$$
\Lambda(s, \chi) = D^{\frac{s}2} \Gamma_{\mathbb R}(s+1)^2
 L(s, \chi) =\Lambda(s,
\chi_1) \Lambda(s, \chi_2) .
$$
Here $\chi_i$ is the quadratic Dirichlet  character associated to $K_i$,
$$
\Gamma_{\mathbb R}(s)= \pi^{-\frac{s}2} \Gamma\left(\frac{s}2\right),
$$
and
$$
\Lambda(s, \chi_i) = |d_i |^{\frac{s}2} \Gamma_{\mathbb R}(s+1) L(s, \chi_i)
$$
is the completed Dirichlet $L$-function of $\chi_i$.
This Eisenstein series  is, up to scalar normalization, Hecke's famous Eisenstein series.

\begin{Prop} \label{Prop:Eisenstein} For $\tau =(\tau_1, \tau_2) \in \mathbb H^2$,
\begin{align*}
E^*(\tau, s, \phi^{\mathcal C}) &= E^*(\tau, s)
 \end{align*}
 where $E^*(\tau, s) =E^*(\tau_1, \tau_2, s)$ is the normalized Hecke Eisenstein series defined in the introduction.
\end{Prop}

\begin{proof}
 Let $I(F)$ be the group of fractional ideals of $F$, and for any nonzero vector $(c,d)\in F^2$ define
$I(c,d) = \co_F c+\co_F d \in I(F)$. Given $\gamma\in \SL_2(F)$ write
\begin{equation}\label{gamma}
\gamma= \left(\begin{matrix} *  & * \\ c & d  \end{matrix}\right)
\end{equation}
and define  $f(\gamma) =I(c,d)$.  Then \cite[Section 1.3]{garrett}  implies that  
$f:\SL_2(F)\map{}I(F)$  induces a bijection
\begin{equation}\label{cusps}
 B(F) \backslash \SL_2(F) /\SL_2(\co_F) \rightarrow \CL(F).
\end{equation}
For each ideal $\mathfrak a \in I(F)$ fix  a $g_\mathfrak a \in \SL_2(F)$ satisfying
$f(g_\mathfrak a) =\mathfrak a$.   The bijectivity of (\ref{cusps}) implies that the function
$$
g_\mathfrak a \SL_2(\co_F) \map{} \{ (c, d) \in \mathfrak a^2:\, I(c, d) =\mathfrak a\}
$$
defined by  $\gamma\mapsto (c, d)$ induces a bijection
\begin{equation*}
B(F)  \backslash B(F) g_\mathfrak a \SL_2(\co_F) \rightarrow
  \{ (c, d) \in \mathfrak a^2:\, I(c, d) =\mathfrak a\}/\co_F^\times.
\end{equation*}

Now we are ready to prove the proposition using the usual unfolding technique. Recall from  (\ref{eq: Eisenstein})
$$
E(\tau, s, \phi^\mathcal C) =(v_1 v_2)^{-\frac{1}2}
 \sum_{\gamma \in B(F) \backslash \SL_2(F)} \Phi^\mathcal C(\gamma g_\tau, s).
$$
For $\gamma \in \SL_2(F)$ written as (\ref{gamma}), let $\widehat\gamma$  denote the  image of $\gamma$ in $\SL_2(\widehat F)$,
and let $\gamma_l =\sigma_l(\gamma)$ be the image of $\gamma$ in $\SL_2(F_{\sigma_l}) =\SL_2(\mathbb R)$.
If we  factor
$$
\widehat{\gamma} =n(b) m(a) k
$$
with $b \in \widehat F$,  $a \in \widehat F^\times$,  and $k \in \SL_2(\widehat{\co}_F)$
then $f(\gamma)$ is the ideal of $F$ associated to $a^{-1}$, and
$$
\widehat{\Phi}^\mathcal C(\widehat{\gamma}, s) = \chi(a) |a|^{s+1} = \chi(f(\gamma)) \norm(f(\gamma))^{s+1}
$$
by Lemma \ref{lem3.0.1}. Next, given $\tau_l=u_l+iv_l\in \mathbb{H}$ there is a $\tau_l' =u'_l+iv'_l\in\mathbb{H}$
such that
$$
\gamma_l g_{\tau_l} = g_{\tau_l'} k_{\theta_l},  \quad  k_{\theta_l} = \kzxz {\cos \theta_l} {\sin \theta_l} {-\sin \theta_l} {\cos \theta_l}.
$$
Writing $c_l =\sigma_l(c)$ and $d_l =\sigma_l(d)$ we have
$$
v_l' = \frac{v_l}{|c_l \tau_l + d_l|^2} \qquad e^{i \theta_l} = \frac{|c_l \tau_l + d_l|}{c_l \tau_l + d_l},
$$
so Lemma \ref{lem3.0.1} implies
$$
\Phi_{\sigma_l}^+(\gamma_l g_{\tau_l}, s) = (v_l')^{\frac{s+1}2} e^{i \theta_l}
   = \frac{v_l^{\frac{s+1}2}}{(c_l \tau_l +d_l) |c_l \tau_l + d_l|^s}.
$$
It now follows that
\begin{align*}
E(\tau, s, \phi^\mathcal C)
 &= \sum_{[\mathfrak a] \in \CL(F)} \sum_{\gamma \in B(F) \backslash  B(F) g_\mathfrak a \SL_2(\co_F)}(v_1 v_2)^{-\frac{1}2} \Phi^\mathcal C(\gamma g_\tau, s)
 \\
 &= \sum_{[\mathfrak a] \in \CL(F)} \chi(\mathfrak a) \norm(\mathfrak a)^{s+1}
  \sum_{\gamma \in B(F) \backslash  B(F) g_\mathfrak a \SL_2(\co_F)}
   \frac{(v_1 v_2)^{\frac{s}2}}{(c \tau +d) |c \tau +d|^s}
   \\
   &= \sum_{[\mathfrak a] \in \CL(F)} \chi(\mathfrak a) \norm(\mathfrak a)^{s+1}
    \sum_{\substack{(c, d) \in \mathfrak a^2/\co_F^\times \\ I(c, d) =\mathfrak a}}
     \frac{(v_1 v_2)^{\frac{s}2}}{(c \tau +d) |c \tau +d|^s}.
\end{align*}
On the other hand, for any $0 \ne (c, d) \in \mathfrak a^2$,  one has $0\ne I(c, d) \subset \mathfrak a$ and so
\begin{align*}
& \sum_{[\mathfrak a] \in \CL(F)} \chi(\mathfrak a) \norm(\mathfrak a)^{s+1}
  \sum_{0 \ne (c, d) \in \mathfrak a^2/\co_F^\times}  \frac{(v_1 v_2)^{\frac{s}2}}{(c \tau +d) |c \tau +d|^s}
  \\
   &= \sum_{[\mathfrak a] \in \CL(F)}
     \sum_{ \substack{ \mathfrak b  \subset \mathfrak a \\ \mathfrak{b}\not=0}}\chi(\mathfrak a) \norm(\mathfrak a)^{s+1}
     \sum_{\substack{(c, d) \in \mathfrak b^2/\co_F^\times \\ I(c, d) =\mathfrak b}}
     \frac{(v_1 v_2)^{\frac{s}2}}{(c \tau +d) |c \tau +d|^s}
     \\
     &= \sum_{[\mathfrak b] \in \CL(F)}  \chi(\mathfrak b) \norm(\mathfrak b)^{s+1}
       \sum_{\substack{(c, d) \in \mathfrak b^2/\co_F^\times \\ I(c, d) =\mathfrak b}}
     \frac{(v_1 v_2)^{\frac{s}2}}{(c \tau +d) |c \tau +d|^s}
      \sum_{\mathfrak c \subset \co_F} \chi(\mathfrak c) \norm(\mathfrak c)^{-s-1}.
\end{align*}
Therefore
$$
\sum_{[\mathfrak a] \in \CL(F)} \chi(\mathfrak a) \norm(\mathfrak a)^{s+1}
  \sum_{0 \ne (c, d) \in \mathfrak a^2/\co_F^\times}  \frac{(v_1 v_2)^{\frac{s}2}}{(c \tau +d) |c \tau +d|^s}
   =L(s, \chi) E(\tau, s, \phi^{\mathcal C})
$$
and the proposition is clear.
\end{proof}

We need some more notation. One has the Fourier expansion:
$$
E^*(\tau, s, \phi^{\mathcal C}) =\sum_{\alpha \in F}
E_\alpha^*(\tau, s, \phi^{\mathcal C}),
$$
where for $\alpha \ne 0$
$$
E_\alpha^*(\tau, s, \phi^\mathcal C)
 = \prod_{\mathfrak p< \infty} W_{\alpha, \mathfrak p}^*(1, s, \phi_\mathfrak p^+) \prod_{l=1}^2
 W_{\alpha, \sigma_l}^*(\tau_l, s, \phi_{\sigma_l}^+),
 $$
 and
 $$
 E_0^*(\tau, s, \phi^{\mathcal C}) =\Lambda(s+1, \chi) (v_1 v_2)^{\frac{s}2} +
  \prod_{\mathfrak p< \infty} W_{0, \mathfrak p}^*(1, s, \phi_\mathfrak p^+) \prod_{l=1}^2
 W_{0, \sigma_l}^*(\tau_l, s, \phi_{\sigma_l}^+).
 $$
 Here (for all $\alpha \in F$)
\begin{align*}
W_{\alpha, \mathfrak p}^*(1, s, \phi_\mathfrak p^+)& = L(s+1, \chi_\mathfrak p)
|D|_\mathfrak p^{-\frac{s+1}2}W_{\alpha, \mathfrak p}(1, s, \phi_\mathfrak p^+)
\\
&=L(s+1, \chi_\mathfrak p)
|D|_\mathfrak p^{-\frac{s+1}2} \int_{F_\mathfrak p} \Phi_\mathfrak p^+(w n(b), s)
\psi_{\mathfrak p}(-\alpha_\mathfrak p b) \, db ,
\end{align*}
 where $db$ is the Haar measure on $F_\mathfrak p$ self-dual with respect to $\psi_{\mathfrak p}$, $w=\kzxz {0} {-1} {1} {0}$,
and
$$
W_{\alpha, \sigma_l}^*(\tau_l, s, \phi_{\sigma_l}^+)
 = v_l^{-\frac{1}2} \Gamma_{\mathbb R}(s+2)
 W_{\alpha, \sigma_l}(g_{\tau_l}, s, \phi_{\sigma_l}^+)
 $$
with
$$
W_{\alpha, \sigma_l}(g_{\tau_l}, s, \phi_{\sigma_l}^+)=\int_{\mathbb R} \Phi_{\sigma_l}^+(w n(b) g_{\tau_l}, s) \psi_{\sigma_l}(- \alpha_l b)\, db.
$$

\subsection{Explicit Calculations} \label{sect:explicit}

We  now record the results of \cite[Proposition 2.1]{YaValue} and   \cite[Proposition 2.2]{YaValue} for the convenience of the
reader.

\begin{Prop} \label{prop3.3}   Let $\mathfrak p$ be  a finite prime  of $F$.
\begin{enumerate}
\item
For all $\alpha \in F$
$$
W_{\alpha, \mathfrak p}^*(1, s, \phi_{\mathfrak p}^+) =|D|_{\mathfrak p}^{-\frac{s}2}
\mathbf{1}_{\mathfrak D_{\mathfrak p}^{-1}}(\alpha) \rho_{\mathfrak p}(\alpha \mathfrak D, s).
$$
Here $\mathfrak D_{\mathfrak p} = \mathfrak D \otimes_{\co_F} \co_{F,\mathfrak p}$, and
$$
\rho_{\mathfrak p}(\mathfrak a, s) =\sum_{r=0}^{\ord_{\mathfrak p} (\mathfrak a)}
(\chi_{\mathfrak p}({\mathfrak p})\norm({\mathfrak p})^{-s})^r.
$$
In particular, $\rho_{\mathfrak p}(\mathfrak a) =\rho_{\mathfrak p}(\mathfrak a, 0)$ and
$$
W_{0, {\mathfrak p}}^*(1, s, \phi_{\mathfrak p}^+)=|D|_{\mathfrak p}^{-\frac{s}2} L(s, \chi_{\mathfrak p}).
$$
\item
For all $\alpha \in \mathfrak D_\mathfrak p^{-1}$
$$
W_{\alpha, {\mathfrak p}}^*(1, 0, \phi_{\mathfrak p}^+)= \rho_{\mathfrak p}(\alpha \mathfrak D).
$$
It is zero if and only if $\chi_{\mathfrak p}(\alpha
\mathfrak D) =-1$, i.e. if and only if  $K/F$ is inert at $\mathfrak p$ and $\ord_{\mathfrak p} (\alpha
\mathfrak D) $ is odd. When this is the case
$$
W_{\alpha, {\mathfrak p}}^{*, \prime} (1, 0, \phi_{\mathfrak p}^+)=-\frac{1}2 \ord_{\mathfrak p} (\alpha
\mathfrak p\mathfrak D)  \log \norm({\mathfrak p}).
$$
\end{enumerate}
\end{Prop}

\begin{proof} (sketch) Let $\psi_{\mathfrak p}' (x)=\psi_{{\mathfrak p}}(\frac{1}{\sqrt D} x)$. Then
$\psi_{\mathfrak p}'$ is an unramified additive character of $F_{\mathfrak p}$, so
\begin{align*}
W_{\alpha, {\mathfrak p}}(g, s, \phi_{\mathfrak p}^+, \psi_{F_{\mathfrak p}})
 &=\int_{F_{\mathfrak p}} \Phi_{\mathfrak p}^+(wn(b) g, s) \psi_{F_{\mathfrak p}}(- \alpha b) \, d_{\psi_{F_{\mathfrak p}}} b
 \\
  &= |D|_\mathfrak p^{\frac{1}2} \int_{F_{\mathfrak p}} \Phi_{\mathfrak p}^+(wn(b) g, s) \psi_{\mathfrak p}'(- \alpha \sqrt D  b) \, d_{\psi_{\mathfrak p}'} b
\\
 &=|D|_\mathfrak p^{\frac{1}2} W_{\alpha \sqrt D, {\mathfrak p}}(g, s, \phi_{\mathfrak p}^+, \psi_{\mathfrak p}').
 \end{align*}
Here we include the additive character in the notation to indicate the dependence of
the Whittaker function on the additive character, and $d_{\psi} b$
is the Haar measure with respect to $\psi$. Now the proposition
follows from \cite[Proposition 2.1]{YaValue}.
\end{proof}

\begin{Prop} \label{prop:infty}
Suppose $\tau=(\tau_1,\tau_2) \in  \mathbb H^2$ and write $\tau_l = u_l+ i v_l$.
\begin{enumerate}
\item
One has
$$
W_{\alpha, \sigma_l}^*(\tau, 0, \phi_{\sigma_l}^+) =\begin{cases}
   -2i e(\sigma_l(\alpha)\tau_l) &\ff \ \sigma_l(\alpha) >0,
   \\
    -i &\ff \ \alpha  =0,
    \\
    0  &\ff \ \sigma_l(\alpha)  <0.
    \end{cases}
    $$
\item
     When $\sigma_l(\alpha) < 0$, one has
    $$
    W_{\alpha , \sigma_l}^{*, \prime}(\tau, 0, \phi_{\sigma_l}^+)
    =-i e(\sigma_l(\alpha)\tau_l) \beta_1(4 \pi |\sigma_l(\alpha)|v_l),
    $$
    where
    $$
    \beta_1(x) =\int_1^\infty e^{-ux} \frac{du}u,    \quad x >0
    $$
    is a partial Gamma function.

  \item
   One has
    $$
    W_{0, \sigma_l}^*(\tau_l, s, \phi_{\sigma_l}^+)
     = v_l^{-\frac{s}2} \Gamma_\mathbb R (s).
     $$
\end{enumerate}
\end{Prop}

\begin{proof}
See \cite[Proposition 15.1]{KRYComp}.
\end{proof}

Simple calculation using the above propositions gives the following  theorem (see
also \cite[Page 215]{GZSingular}).
\begin{Thm} \label{theo3.1.3}
One has $E^*(\tau, 0, \phi^{\mathcal C}) =0$, and
$$
E^{*, \prime}(\tau, 0, \phi^{\mathcal C})
 = \sum_{\alpha \in \mathfrak D^{-1}} a_\alpha(v_1, v_2) q^\alpha
 $$
 where $v_l$ is the imaginary part of $\tau_l$, $q^\alpha =e( \sigma_1(\alpha) \tau_1 + \sigma_2(\alpha)
 \tau_2)$, and $a_\alpha(v_1, v_2)$ is  as follows.
\begin{enumerate}
 \item
Assume   $\alpha$ is totally positive, and recall that  $\Diff(\alpha)$ has odd cardinality.
If $|\Diff(\alpha)| >1$, then $a_\alpha(v_1, v_2) =0$. If $\Diff(\alpha)=\{\mathfrak p\}$, then
 $a_\alpha= a_{\alpha}(v_1, v_2)$ is independent of $v_1, v_2$ and is equal to
 $$
 a_\alpha =2\ord_{\mathfrak p} (\alpha \mathfrak p \mathfrak D) \rho(\alpha
 \mathfrak D \mathfrak p^{-1}) \log \norm(\mathfrak p).
 $$

\item
 When $\sigma_k(\alpha) > 0 > \sigma_l(\alpha) $ with $\{k, l\}
=\{1, 2\}$, one has
$$
a_\alpha(v_1, v_2) =2 \rho(\alpha \mathfrak D) \beta_1(4 \pi
|\sigma_l(\alpha)| v_l).
$$

\item
 The constant term is
$$
a_0(v_1, v_2) =2 \Lambda(0, \chi) \left( -\frac{\Lambda'(0,
\chi)}{\Lambda(0,\chi)} + \frac{1}2 \log (v_1 v_2)\right).
$$

\item
When $\alpha$ is totally negative, $a_\alpha(v_1, v_2) =0$.
\end{enumerate}
\end{Thm}

\begin{proof} (sketch) Recall that for $\alpha \in F^\times$
$$
E_\alpha^*(\tau, s, \phi^\mathcal C)=\prod_{\mathfrak p <\infty} W_{\alpha, \mathfrak p}^*(1,s, \phi_\mathfrak p^+)
   \prod_{l=1}^2 W_{\alpha, \sigma_l}^*(\tau_l, s, \phi_{\sigma_l}^+),
$$
which is zero unless $\alpha \in \mathfrak D^{-1}$ by Proposition  \ref{prop3.3}. Assume $\alpha \in \mathfrak D^{-1}$ is totally positive.
Proposition \ref{prop3.3} implies $W_{\alpha, \mathfrak p}^*(1, 0, \phi_\mathfrak p^+) =0$ if $\mathfrak p \in \Diff(\alpha)$, and
so $E_{\alpha}^{*, \prime}(\tau, 0, \phi^{\mathcal C}) =0$ when $|\Diff(\alpha)| >1$.
Still assuming that $\alpha$ is totally positive, when $\Diff(\alpha) =\{\mathfrak p\}$  one has, by
Propositions \ref{prop3.3} and \ref{prop:infty},
\begin{align*}
E_{\alpha}^{*, \prime}(\tau, 0, \phi^{\mathcal C})
 &= W_{\alpha, \mathfrak p}^{*, \prime} (1, 0, \phi_\mathfrak p^+)
  \prod_{\mathfrak q \ne \mathfrak p} W_{\alpha, \mathfrak q}^*(1,0, \phi_\mathfrak q^+)
   \prod_{l=1}^2 W_{\alpha, \sigma_l}^*(\tau_l, 0, \phi_{\sigma_l}^+)
 \\
  &= 2 \ord_{\mathfrak p}(\alpha \mathfrak p \mathfrak D) \log \norm(\mathfrak p)  \cdot
  \rho(\alpha \mathfrak D \mathfrak p^{-1}) \cdot q^\alpha
\end{align*}
as claimed. Here we have used the fact
$
\rho_{\mathfrak p}(\alpha  \mathfrak D \mathfrak p^{-1})=1$ when $\Diff(\alpha) =\{\mathfrak p\}$, and
$$
\rho(\alpha\mathfrak D \mathfrak p^{-1})=\prod_{\mathfrak q<\infty} \rho_{\mathfrak q}(\alpha\mathfrak D \mathfrak p^{-1}).
$$
 The other  cases are similar and left to the reader.
\end{proof}

 \begin{proof}[Proof of Theorems  \ref{theo1.2} and \ref{maintheo}]
Fix a totally positive $\alpha\in F^\times$.   Assuming that $\alpha\in \mathfrak{D}^{-1}$
and $\Diff(\alpha)=\{\mathfrak{p}\}$,  Theorem \ref{theo1.2} is just a restatement of the first claim of
Theorem \ref{theo3.1.3}.   Note we have used the fact that $\Diff(\alpha)=\{\mathfrak{p}\}$ implies that $\mathfrak{p}$
is inert in $K$.  This can only happen if  the prime $p$ of $\Q$ below $\mathfrak{p}$ is nonsplit in both $K_1$ and $K_2$,
which implies   $\norm(\mathfrak{p})=p$ by Remark \ref{Rem:prime description}.
If $\Diff(\alpha)=\{\mathfrak{p}\}$ but $\alpha\not\in \mathfrak{D}^{-1}$ then
$\rho(\alpha\mathfrak{D}\mathfrak{p}^{-1})=0$, and so $a_\alpha(v_1,v_2)=0$ by the first claim of Theorem \ref{theo3.1.3}.
If $|\Diff(\alpha)| >1$ then, again by the first claim of Theorem \ref{theo3.1.3},
$a_\alpha(v_1,v_2)=0$.  This completes the proof of Theorem \ref{theo1.2}.
 Theorem \ref{maintheo} follows immediately from  Theorems \ref{geometric theorem} and \ref{theo1.2}.
\end{proof}

\subsection{A conceptual proof of Theorem \ref{maintheo}}
\label{sect:concept}

 In this subsection we give a more conceptual proof of Theorem \ref{maintheo} which
 is based on the Siegel-Weil formula (as opposed to the explicit calculation of both sides of the stated equality).

 For a finite prime ${\mathfrak p}$ of $F$ inert  in $K$, let $W_{\mathfrak p}^\pm$ be the
 binary quadratic space  $K_{\mathfrak p}$ over $F_{\mathfrak p}$ with quadratic form
 $$
Q_\mathfrak p^+(x) = \frac{1}{\sqrt D} x \bar x, \quad  Q_{\mathfrak p}^- (x) =\frac{\pi_{\mathfrak p}}{\sqrt D} x \bar x,
 $$
 where $\pi_{\mathfrak p}$ is a uniformizer of $F_{\mathfrak p}$. Notice that  $\mathcal C_\mathfrak p \cong W_{\mathfrak p}^+ $.
 Let $(W^{({\mathfrak p})}, Q^{({\mathfrak p})})$ be the global (totally positive definite)
 $F$-quadratic space obtained from $\mathcal C$ by changing
 $\mathcal C_{\mathfrak p}=W_{\mathfrak p}^+$ to $W_{\mathfrak p}^-$ and leaving the other local quadratic
 spaces $\mathcal C_{v}=W_{v}^+$ unchanged.

\begin{Lem}  \label{lem3.1.3}\
\begin{enumerate}
\item
  Let  $(\mathbf E_1, \mathbf E_2) \in[ \mathcal X(\mathbb F_p^{\alg})]$ be a supersingular CM pair, and
let ${\mathfrak p}$ be the reflex prime of $(\mathbf E_1, \mathbf E_2)$. Then there is an isomorphism of $\A_F$-quadratic spaces
$$
({V}(\mathbf E_1, \mathbf E_2)\otimes_F \A_F, \deg_{\CM})  \cong ({W}^{({\mathfrak p})} \otimes_F \A_F
, Q^{({\mathfrak p})})
$$
that maps  $\widehat{L}(\mathbf E_1, \mathbf E_2) $ onto
$\widehat{\co}_{K}$. In particular, there is an isomorphism of  $F$-quadratic spaces
$$
\big( V(\mathbf E_1, \mathbf E_2), \cmdeg\big) \iso  ( W^{({\mathfrak p})} , Q^{(\mathfrak{p})} \big).
$$

\item
 If $(\mathbf E_1, \mathbf E_2, j) \in [ \mathcal X_{\alpha}(\mathbb F_p^{\alg})]$  then
$\Diff(\alpha) =\{ {\mathfrak p}\}$, where ${\mathfrak p}$ is  the reflex
prime of $(\mathbf E_1, \mathbf E_2)$. In particular, if $|\Diff(\alpha)| >1$, then $\mathcal X_{\alpha}$ is empty.
\end{enumerate}
\end{Lem}

\begin{proof} Part (1) follows from Theorem  \ref{Thm:quadratic model}, together with the Hasse-Minkowski Theorem.
Next,  $(\mathbf E_1, \mathbf E_2, j) \in [\mathcal X_{\alpha}(\mathbb F_p^{\alg})]$ implies that there is a  $j \in
L(\mathbf E_1,\mathbf E_2)$  with $\deg_{\CM}( j) =\alpha$. By (1) there is $ z \in \widehat{K}$ such that $Q^{(\mathfrak p)}(z)
=\alpha$. This implies that $\Diff(\alpha)=\{ \mathfrak p\}$.
\end{proof}

\begin{Prop} \label{Prop:counting}
Let $\mathfrak p$  be the reflex prime of  a supersingular CM pair $(\mathbf E_1, \mathbf E_2) \in \mathcal X(\mathbb F_p^{\alg})$,
and let
$$
\phi^{(\mathfrak p)} = \mathbf 1_{\widehat{\co}_K}  \otimes  \phi_{\sigma_1}^+ \otimes  \phi_{\sigma_2}^+ \in
\Sch(W^{({\mathfrak p})} \otimes_F \A_F) .
$$
 Then for every totally positive $\alpha \in F$ one has
$$
E_\alpha^*(\tau, 0, \phi^{(\mathfrak p)}) q^{-\alpha} =\frac{ 8}{\mathbf w_1 \mathbf w_2}
 \sum_{(\mathfrak a_1, \mathfrak a_2) \in \Gamma}
  \sum_{\substack{j \in L(\mathbf E_1 \otimes \mathfrak a_1, \mathbf E_2 \otimes \mathfrak a_2) \\ \deg_{\CM}(j)=\alpha}}  
  1.
$$
Here
$$
E^*(\tau, s , \phi^{(\mathfrak p)})=\Lambda(s+1, \chi) E(\tau, s,  \phi^{(\mathfrak p)})
$$
is a (non-holomorphic)  Hilbert modular form of weight $1$ defined as in (\ref{eq: Eisenstein}).
\end{Prop}

We remark that the left hand side in the formula  depends only on $\mathfrak{p}$, and not on the pair $(\mathbf E_1, \mathbf E_2)$.

\begin{proof}
As $W^{(\mathfrak{p})}$ is a $K$-vector space the algebraic group $S$ defined in
 Section \ref{sect2.3} acts on $W^{(\mathfrak{p})}$, and this action identifies
$S \iso \hbox{Res}_{F/\mathbb Q} \SO(W^{({\mathfrak p})})$.   Let
\begin{equation*}
 \theta(g, h, \phi^{({\mathfrak p})}) =\sum_{j \in W^{({\mathfrak p})}} \omega_{W^{({\mathfrak p})}, \psi_F}
 (g)\phi^{({\mathfrak p})}(h^{-1} j)
 \end{equation*}
 be the theta kernel; here $g \in \SL_2(\A_F)$ and  $h \in S(\A)$.  Let
 $$
\theta(g, \phi^{({\mathfrak p})}) = \int_{[S]}
\theta(g, h, \phi^{({\mathfrak p})}) \,  dh
$$
be the associated theta integral, where  $[S] =S(\Q) \backslash S(\A)$, and  $dh$ is an 
$S(\A)$-invariant  measure on $[S]$. Then
$$
\theta(\tau, \phi^{({\mathfrak p})}) =(v_1 v_2)^{-\frac{1}2} \theta(g_\tau, \phi^{({\mathfrak p})})
$$
is a  holomorphic Hilbert modular form of weight $1$. Moreover, the Siegel-Weil formula 
(\cite{Weil}, \cite{KuBintegral}) asserts
$$
E(\tau, 0, \phi^{(\mathfrak p)}) =C_1\theta(\tau, \phi^{(\mathfrak p)})
$$
 for some constant $C_1$ depending  on the measure $dh$ (in fact $C_1 =2 \cdot  \vol([S])^{-1}$, but we won't need this).
Next, by Lemma \ref{lem3.1.3}, there is an isomorphism  
$\widehat{V}(\mathbf E_1, \mathbf E_2) \cong \widehat{W}^{(\mathfrak p)}$ that maps
$\widehat{L}(\mathbf E_1, \mathbf E_2)$ onto $\widehat{\co}_K$. It follows that
$$
 \int_{[S]} \sum_{\substack{j \in W^{(\mathfrak{p})} \\ Q^{(\mathfrak{p})}(j) =\alpha} }
\mathbf \phi^{(\mathfrak{p})} (h^{-1} j)\, dh \\
= \int_{[S]} \sum_{\substack{j \in V(\mathbf E_1, \mathbf E_2) \\ \deg_{\CM}(j) =\alpha} }
\mathbf 1_{\widehat{L}(\mathbf E_1, \mathbf E_2)}(h^{-1} j)\, dh
$$
is the $\alpha$-th Fourier coefficient of $\theta(\tau, \phi^{(\mathfrak{p})})$.

Recall from Section \ref{sect2.3} that  $\Gamma$ acts on $[\mathcal X(\mathbb F_{p}^\alg)]$.  Using the map
(\ref{torus ideals}) we obtain an action of $T(\widehat{\Q})$ on $[\mathcal X(\mathbb F_{p}^\alg)]$,
and this action factors through the map  $\eta: T(\widehat{\Q})\map{} S(\widehat{\Q})$ of  (\ref{unitary map}). It is easy to see
that $\mathbf 1_{\widehat{L}(\mathbf E_1, \mathbf E_2)}$ is invariant under $V=\eta(U)$, so there is a constant $C$,
independent of $\alpha$, such that
\begin{align*}
E_\alpha(\tau, 0, \phi^{(\mathfrak p)}) q^{-\alpha}
 &= C \sum_{h \in S(\Q) \backslash S(\widehat{\Q})/\eta(U)} \sum_{\substack{j \in V(\mathbf E_1, \mathbf E_2) \\ \deg_{\CM}(j) =\alpha} } \mathbf 1_{\widehat{L}(\mathbf E_1, \mathbf E_2)}(h^{-1} j)
 \\
  &= C \sum_{t \in T(\Q) \backslash T(\widehat{\Q})/U}
  \sum_{\substack{j \in V(\mathbf E_1, \mathbf E_2) \\ \deg_{\CM}(j) =\alpha} } \mathbf 1_{t\action \widehat{L}(\mathbf E_1, \mathbf E_2)}( j)
  \\
   &=C \sum_{(\mathfrak a_1, \mathfrak a_2)\in \Gamma} \sum_{\substack{j \in L(\mathbf E_1 \otimes \mathfrak a_1, \mathbf E_2\otimes \mathfrak a_2) \\ \deg_{\CM}(j) =\alpha}} 1.
\end{align*}
The last identity follows from Proposition \ref{Prop:index} and the discussion following Remark \ref{four orbits}.
Notice that  the Eisenstein series has  constant term $E_0(\tau, 0,  \phi^{(\mathfrak p)})=2$
(see for example \cite[Theorem 1.2]{YaValue}).  Taking $\alpha =0$ on both sides, one sees that
$$
C=\frac{2}{|\Gamma|} = \frac{2}{h_1 h_2}
$$
where $h_i$ is the  class number of $K_i$.  Now the proposition follows from the class number formula
$$
\Lambda(1, \chi) = \Lambda(1, \chi_1) \Lambda(1, \chi_2) = \frac{ 4 h_1 h_2}{\mathbf w_1 \mathbf w_2},
$$
in which $\chi_i$ is the quadratic Dirichlet character associated to $K_i/\Q$.
Alternatively, one can find $C$ by tracking the Haar measures involved.
\end{proof}

  \begin{Prop}  \label{lem3.11}
  Assume  $ \mathfrak p \in  \Diff(\alpha) $.
\begin{enumerate}
\item
 One has $W_{\alpha, \mathfrak p}^*(1, s, \phi_{\mathfrak p}^{(\mathfrak p)}) =0$ unless  $\alpha \in  \mathfrak D_{\mathfrak p}^{-1}$. In such a case, one has
  $$
  W_{\alpha, \mathfrak p}^*(1, 0, \phi_{\mathfrak p}^{(\mathfrak p)}) =-1.
  $$
\item
 One has
  $$
  W_{\alpha, \mathfrak p}^{*, \prime}(1, 0, \phi_{\mathfrak p}^{\mathcal C})
       = \nu_\mathfrak p(\alpha)   {W_{\alpha, \mathfrak p}^*(1, 0, \phi_{\mathfrak p}^{(\mathfrak p)})} \log \norm(\mathfrak p).
  $$
  Here
  $$\nu_\mathfrak p(\alpha)= \frac{1}2 \ord_{\mathfrak p}(\alpha \mathfrak p \mathfrak D)
  $$
as in Proposition  \ref{Prop:local length}.
  \end{enumerate}
  \end{Prop}

  \begin{proof} The first equality of (1) follows from \cite[Proposition 2.2]{YaValue}
  (and the argument in the proof of Proposition \ref{prop3.3}). The rest follows from Proposition \ref{prop3.3}.
  Notice that both sides of the stated equality of (2) are zero if $\alpha \notin \mathfrak D_\mathfrak p^{-1}$.
  \end{proof}

  \begin{proof} [Proof of Theorem \ref{maintheo}] Assume first that $\alpha \in F^\times$ is totally positive and
  $ \Diff(\alpha) =\{ \mathfrak p\}$. Let $p$ be the rational prime below $\mathfrak p$.  As $\mathfrak{p}$ is inert in $K$, $p$
  is nonsplit in both $K_1$ and $K_2$.  Given a (necessarily supersingular) CM pair  $(\mathbf E_1, \mathbf E_2)$ over $\F_p^\alg$,
   write $\mathfrak p (\mathbf E_1, \mathbf E_2)$ for its reflex prime.
  Recall from Theorem \ref{Thm:local length} that for any geometric point  $x\in  \mathcal X_\alpha(\mathbb F_p^\alg)$
  representing a triple $(\mathbf{E}_1,\mathbf{E}_2,j)$ we have   $\mathfrak p(\mathbf E_1, \mathbf E_2)=\mathfrak p$  and
  $$
 \length( \co^\mathrm{sh}_{  \mathcal{X}_\alpha, x  } )= \nu_\mathfrak p(\alpha)=
  \frac{1}2 \ord_{\mathfrak p}(\alpha \mathfrak p \mathfrak D) .
 $$
 Therefore
\begin{align*}
\deg(\mathcal X_\alpha)
&= \nu_\mathfrak p(\alpha) \log (p) \sum_{(\mathbf E_1, \mathbf E_2, j) \in [\mathcal X_{\alpha}(\mathbb F_p^\alg)]}
 \frac{1}{|\Aut(\mathbf E_1, \mathbf E_2, j)|} \\
 &= \nu_\mathfrak p(\alpha) \log (p) \sum_{(\mathbf E_1, \mathbf E_2) \in [\mathcal X(\mathbb F_p^\alg)]}
  \sum_{ \substack{ j \in L(\mathbf E_1, \mathbf E_2) \\ \cmdeg(j)=\alpha }} \frac{1}{|\Aut(\mathbf E_1, \mathbf E_2)|}.
 \end{align*}
Lemma \ref{C stab} then implies
\begin{align*}
\deg (\mathcal X_\alpha) &= \frac{\nu_\mathfrak p(\alpha)}{\mathbf w_1 \mathbf w_2} \log( p) \sum_{\substack{ (\mathbf E_1, \mathbf E_2) \in [\mathcal X(\mathbb F_p^\alg)] \\  \mathfrak p(\mathbf E_1, \mathbf E_2) =\mathfrak  p}} \sum_{\substack{ j\in L(\mathbf E_1,  \mathbf E_2) \\ \deg_{\CM} j =\alpha}} 1.
\end{align*}
By Propositions \ref{Prop:naked count} and Lemma \ref{C stab}, $\Gamma$ acts on the set of
$(\mathbf E_1, \mathbf E_2) \in  [\mathcal X(\mathbb F_p^\alg)] $ with reflex prime $\mathfrak p$ freely with two orbits,
so  Proposition \ref{Prop:counting} implies
$$
\deg (\mathcal X_\alpha) = \frac{1}4  \nu_\mathfrak p(\alpha) \log (p ) E_\alpha^*(\tau, 0,  \phi^{(\mathfrak p)}) q^{-\alpha}.
$$
Next,  Proposition  \ref{lem3.11} implies
\begin{align*}
E_{\alpha}^{*, \prime}(\tau, 0, \phi^{\mathcal C})
 &=W_{\alpha, \mathfrak p}^{*, \prime}(1, 0, \phi_{\mathfrak p}^{\mathcal C})
 \prod_{\mathfrak q \ne \mathfrak p} W_{\alpha, \mathfrak q}^{*}(1, 0, \phi_{\mathfrak q}^{\mathcal C})
 \prod_{l=1}^2 W_{\alpha, \sigma_l}^*(\tau_l, 0, \phi_{\sigma_l}^\mathcal C)
 \\
  &=
  \nu_\mathfrak p(\alpha) \log \norm(\mathfrak p) \cdot E_\alpha^*(\tau, 0,  \phi^{(\mathfrak p)}).
\end{align*}
Notice that $\norm(\mathfrak p) =p$ by Remark \ref{Rem:prime description},  so we have
$$
4 \deg (\mathcal X_\alpha)  \cdot q^\alpha = E_\alpha^{*, \prime}(\tau, 0, \phi^{\mathcal C}) = E_\alpha^{*, \prime}(\tau, 0)
$$
by
 Proposition \ref{Prop:Eisenstein}, as claimed in  Theorem \ref{maintheo}.

Next, if $|\Diff(\alpha)| >1$ then $\ord_{s=0} E_\alpha^*(\tau, s,  \phi^\mathcal C) >1$ by Proposition \ref{prop3.3},
while  $\mathcal X_\alpha$ is empty by Theorem \ref{Thm:local length}. Thus both sides of the desired equality are zero.
\end{proof}

One advantage of the conceptual proof is that  Theorem \ref{maintheo} can be proved for $\Diff(\alpha)=\{\mathfrak{p}\}$
by only doing local calculations at  $\mathfrak{p}$.  For example one does not need to compute (as in Section \ref{ss:quadratic})
the  local structure of  $(V(\mathbf{E}_1,\mathbf{E}_2) , \cmdeg)$ at primes other than  $\mathfrak{p}$, nor
any of the orbital integrals of Section \ref{ss:orbital},
nor does one need to  compute  local Whittaker functions at primes other than $\mathfrak{p}$.  The obvious
disadvantage is that one does not obtain the explicit values for  $\deg(\mathcal{X}_\alpha)$ and $a_\alpha$ of
Theorems \ref{geometric theorem} and \ref{theo1.2}.

\end{document}